\documentclass[a4paper,10pt,reqno]{amsart}
\usepackage[body={13cm,21cm}]{geometry}
\usepackage[initials]{amsrefs}

\usepackage{amssymb,amscd}
\usepackage[all]{xy}
\usepackage{remark}
\usepackage{enumerate}

\newtheorem{thm}{Theorem}[section]
\newtheorem{cor}[thm]{Corollary}
\newtheorem{theorem}[thm]{Theorem}
\newtheorem{lem}[thm]{Lemma}
\newtheorem{prop}[thm]{Proposition}
\newtheorem{conjecture}[thm]{Conjecture}
\newremark{definition}[thm]{Definition}
\newremark{example}[thm]{Example}
\newremark{rem}[thm]{Remark}
\numberwithin{equation}{section}

\DeclareMathOperator{\Res}{Res} 
\DeclareMathOperator{\res}{res}

\DeclareMathOperator{\Br}{Br}
\DeclareMathOperator{\Div}{Div}
\DeclareMathOperator{\Pic}{Pic}

\DeclareMathOperator{\Gal}{Gal}
\DeclareMathOperator{\Spec}{Spec}
\DeclareMathOperator{\rank}{rank}
\DeclareMathOperator{\inv}{inv}
\DeclareMathOperator{\Hom}{Hom}
\DeclareMathOperator{\Ker}{Ker} 

\DeclareMathOperator{\Sel}{Sel}

\newcommand{\cO}{\mathcal O}
\def\Q{{\mathbb Q}}
\def\Z{{\mathbb Z}}

\def\N{{\mathbb N}}
\def\G{{\mathbb G}}

\def\A{{\mathbb A}}
\def\go{\mathfrak{o}}
\def\T{{\mathbf T}}
\def\X{{\mathbf X}}
\def\Y{{\mathbf Y}}

\DeclareFontFamily{U}{wncy}{}
\DeclareFontShape{U}{wncy}{m}{n}{%
<5>wncyr5%
<6>wncyr6%
<7>wncyr7%
<8>wncyr8%
<9>wncyr9%
<10>wncyr10%
<11>wncyr10%
<12>wncyr6%
<14>wncyr7%
<17>wncyr8%
<20>wncyr10%
<25>wncyr10}{}
\DeclareMathAlphabet{\cyr}{U}{wncy}{m}{n}

\DeclareFontEncoding{OT2}{}{} 
  \newcommand{\textcyr}[1]{%
    {\fontencoding{OT2}\fontfamily{wncyr}\fontseries{m}\fontshape{n}%
     \selectfont #1}}

\newcommand{\Sha}{{\mbox{\textcyr{Sh}}}}

\begin{document}

\title[Very strong approximation]
{Very strong approximation for certain algebraic varieties}

\author{Qing Liu}
\author{Fei Xu}

\begin{abstract} Let $F$ be a global field. In this work, we show that the Brauer-Manin 
condition on adelic points for subvarieties of a torus $T_F$ over $F$ cuts out exactly 
the rational points, if either $F$ is a function field or, if $F=\Q$ and $T_F$ is split.
As an application, we prove a conjecture of Harari-Voloch over global function fields 
which states, roughly speaking, that on any rational hyperbolic curve, the local integral 
points  with the Brauer-Manin condition are the global integral points. 
Finally we prove for tori over number fields a theorem of Stoll on adelic 
points of zero-dimensional subvarieties in abelian varieties.
\end{abstract}

\address{Universit\'e de Bordeaux, Institut de Math\'ematiques de Bordeaux,
351 cours de la Li\-b\'eration, 33405 Talence, France}
\email{Qing.Liu@math.u-bordeaux1.fr}

\address{School of Mathematical Sciences, Capital Normal University, Beijing 100048, China}
\email{xufei@math.ac.cn}



\date{\today}

\maketitle

\section{Introduction}\label{intro}

Since the rational points of an algebraic variety over a global field satisfy 
the Brauer-Manin conditions,  it is natural to ask whether, conversely, 
these conditions determine the rational points. There is some recent progress 
for projective curves and abelian varieties (\cite{PV}, see also \cite{St}). On the other hands, 
Colliot-Th\'el\`ene and the second named author introduced in \cite{CTX} 
the study of integral points for homogeneous spaces of linear algebraic groups 
in relation with strong approximation properties. Afterward, Harari and Voloch \cite{HaVo} 
considered the integral points on hyperbolic curves satisfying some Brauer-Manin conditions. 
In particular they conjectured that these conditions determine the integral points of
rational hyperbolic curves.  In the present paper, the main purpose is first to 
pursuit the study of integral points and strong approximation properties for general
algebraic varieties with focus on subvarieties of torus. Secondly, we solve 
Harari-Voloch conjecture for global function fields.

Now let us fix some notation and give some definition.
Let $F$ be a global field, $\Omega_F$ be the set of all primes in $F$,
$\infty_F$ be the set of all Archimedean primes of $F$
($\infty_F=\emptyset$ if $F$ is a function field).
We will denote by $\go_S$ the ring of $S$-integers for a \emph{non-empty} finite subset $S$ of
$\Omega_F$ containing $\infty_F$.
If $F$ is a number field, we denote by $\go_F$ the ring of integers of
$F$. For each finite prime $v\in \Omega_F$,
$\go_{F,v}$ {denotes the discrete valuation ring of $F$ associated to $v$,
${\go}_v$ its completion}, $F_v=\mathrm{Frac}({\go}_v)$.
If $v\in \infty_F$, we put $\go_v={\go}_{F,v}=F_v$.

Let $\mathbb A_F$ be the adelic ring of $F$.
We denote by $\mathbb A_F^S$
the $S$-adeles obtained by projecting $\mathbb A_F$ to
$\prod_{v\not\in S} F_v$ endowed with the induced adelic
topology and have
$$\A_F=(\prod_{s\in S} F_v)\times \A_F^S$$
for any finite subset $S$ of $\Omega_F$ containing $\infty_F$.

An \emph{algebraic variety over $F$} will be a separated scheme of
finite type over $F$. Let $X_F$ be an algebraic variety.
{For all $v\in\Omega_F$, $X_F(F_v)$ is endowed
with the topology induced by the topology of $F_v$.}
The set of adelic points
$$X_F(\mathbb A_F):=\mathrm{Mor}_F(\Spec \mathbb A_F, X_F)\subset\prod_{v\in\Omega_F} X_F(F_v)$$
can be canonically endowed with a locally compact Hausdorff topology
(\cite{Conrad}, \S 4).  When $X_F$ is affine, $X$ is a closed subvariety of an affine space. Then the adelic topology of  
$X_F(\mathbb A_F)$ is induced by the product of usual adelic topology under the closed immersion.

Let $$\Br(X_F)=H^2_{\text{\'et}}(X_F, \mathbb G_m)  \quad  \text{and} \quad \xi\in \Br(X_F). $$
For any $(x_v)_{v\in \Omega_F}\in X_F(\A_F)$, one has
$\inv_v(\xi(x_v))=0$ for almost all $v\in \Omega_F$ (see \cite{Sk} \S 5.2), where
$$\inv_v : \Br(F_v)\to\mathbb Q/\mathbb Z$$ is the local invariant
map. For any subgroup $H$ of $\Br(X_{F})$, define
$$X_F(\mathbb A_F)^{H}= \{ (x_v)_v\in X_F(\mathbb A_F):
  \sum_{v\in \Omega_F }\inv_v(\xi(x_v))=0 \ \text{ for all $\xi\in H$}
  \}.$$
The Artin reciprocity law implies that
$$ X_F(F) \subseteq X_F(\mathbb A_F)^{\Br(X_F)} \subseteq X_F(\mathbb A_F). $$
For each $\xi$, $\inv_v(\xi(-))$ is locally constant,
{hence continuous}, on $X_F(F_v)$.
This implies that $X_F(\mathbb A_F)^{H}$ is a closed subset
of $X_F(\mathbb A_F)$.
\medskip

Part of the following definition is taken from \cite{CTX11}, Definition 2.1.

\begin{definition}\label{def-BM} Let $S$ be a finite subset of $\Omega_F$ and
let
$$P_S: \ X_F(\mathbb A_F) \longrightarrow X_F(\mathbb A_F^S)$$
be the projection map.
We will say that \emph{the strong approximation property with
Brauer-Manin obstruction off $S$ holds for $X_F$} or
that \emph{$X_F$ satisfies the strong approximation property with
Brauer-Manin obstruction off $S$} (SAP-BM off $S$ for short) if $X_F (F)$,
under the diagonal map, is dense in $P_S[X_F(\mathbb
A_{F})^{\Br(X_F)}]$ with adelic topology.

We will say that \emph{the very strong approximation property with
Brauer-Manin obstruction off $S$ holds for $X_F$} or
that \emph{$X_F$ satisfies the very strong approximation property with
Brauer-Manin obstruction off $S$} (abbreviated to VSAP-BM off $S$) if
$$X_F (F) = P_S[X_F(\mathbb A_{F})^{\Br(X_F)}]. $$
\end{definition}

In the classical  strong approximation examples
({\it e.g.}, affine spaces and simply connected semi-simple linear
algebraic groups),
the Brauer groups are always constant ({\it i.e.}, the canonical
map $\Br(F)\to \Br(X_F)$ is surjective) and $X_F(\A_F)^{\Br(X_F)}=X_F(\A_F)$.
So the above definition extends the classical situations.
\medskip

The first main result of this work deals with VSAP-BM and can be
regarded as an analogue of Theorem D in \cite{PV} for tori.

\begin{theorem}{\rm (Corollary~\ref{reb})}
Let $X_F$ be a subvariety of a torus {$T_F$} over a global field $F$.
Then $X_F$ satisfies VSAP-BM off $S$ for any finite subset $S$ of
$\Omega_F$ containing $\infty_F$ if one of the following
conditions is satisfied:
\begin{enumerate}
\item $F$ is a global function field (then $\infty_F=\emptyset$);
\item $F=\Q$ or an imaginary quadratic field and $T_F$ is split over $F$.
\end{enumerate}
\end{theorem}

For general algebraic varieties, we have the following result.

\begin{theorem}{\rm (Corollary~\ref{rev})} Any algebraic variety over $\Q$ or an imaginary quadratic field always contains an dense open subvariety which satisfies  VSAP-BM off $S$ for any finite subset $S$ of
$\Omega_F$ containing $\infty_F$.  
\end{theorem}

As an application, we prove Harari-Voloch's conjecture
over a global function field by using \cite{Sun}, Theorem 1.

\begin{theorem}\label{m}{\rm (Corollary \ref{true})}
Let $F$ be a global function field and $X_F$ be an open subset
of $\mathbf P^1_F$ with complementary of degree $\ge 3$ as a reduced
separable divisor. Then
 $$ X_F(F) = X_F({\mathbb A}_F^S)^{B_S(X_F)}, $$
 where
 $$B_S(X_F)= \Ker[\Br(X_F) \to \prod_{v\in S}
 \Br(X_{F_v})/\Br(F_v)] $$
for any finite subset $S$ of $\Omega_F$.
\end{theorem}

\noindent We prove that the same result is also true over $\Q$ when $S=\infty_{\Q}$ and $\mathbf P^1_{\Q} \setminus X_\Q$ contains more than a single point in Corollary~\ref{hv-ex2}.  
\medskip

The next theorem  can be regarded as an analogue for tori of \cite{St}, Theorem 3.11. 
After this paper was written, we learned that C.-L. Sun proved an equivalent statement 
(\cite{Sun2}, Theorem 1) by a slightly different method.

\begin{thm}{\rm (Theorem \ref{inters})} If $Z$ is a finite subscheme
  of a torus $T_F$ over a number field $F$, then for any
finite subset $S$ of $\Omega_F$ containing $\infty_F$, we have
 $$ \overline{T_F(F)}^{S} \cap Z(\A_F^S) = Z(F) $$
where $\overline{T_F(F)}^S$ is the topological closure of $T_F(F)$ inside $T_F(\A_F^S)$.
\end{thm}

Note that if this result is true for any curve contained in $T_F$, then Harari-Voloch conjecture 
holds over number fields (Proposition~\ref{eqhv} (1)). As an application of the above theorem, we prove in 
Proposition~\ref{quasi-finite} that if $f: X_F\to Y_F$ is a quasi-finite morphism of rational hyperbolic curves and
if Harari-Voloch conjecture is true for $Y_F$, then it is true for $X_F$. 

\section{Integral models and strong approximation}

Let $F, S$ and $X_F$ be as in \S~\ref{intro}. 
The aim of this section is first to show that the subsets of $X_F(\A_{F}^S)$
of the form $\prod_{v\notin S}\X(\go_v)$, when $\X$ runs through the 
integral models of $X_F$ over $\go_S$, form a topology basis for 
$X_F(\A_F^S)$ (Corollary~\ref{U-model}). Then we apply this result to 
relate strong approximation of adelic points to that of integral 
points (Theorem~\ref{Diophantine-sap}).

\subsection{Integral models} 
We start with some basic definitions and technical preliminary results.

\begin{definition} Let $S$ be a non-empty finite subset of
$\Omega_F$ containing $\infty_F$.
An \emph{integral model (or a model) of $X_F$} over
$\go_S$ is a scheme $\X$ faithfully flat of finite type
and separated over $\go_S$ endowed with an isomorphism
${\X}\times_{\go_S} F\cong X_F$ over $F$.
An \emph{integral point on $\X$} is a section $\in\X(\go_S)$.
\end{definition}

Define $\mathbb A_{F,S}=\prod_{v\in S} F_v\times \prod_{v\notin S}
{\go}_v\subset \mathbb A_F$.
Then $\X(\mathbb A_{F,S})\subseteq X_F(\mathbb A_F)$ and
it is known that the canonical map
$$ \X(\mathbb A_{F,S})\to \prod_{v\in S} \X(F_v) \times \prod_{v\notin S}
\X({\go}_v)=\prod_{v\in S} X_F(F_v) \times \prod_{v\notin S}\X({\go}_v)$$
is a bijection (\cite{Conrad}, Theorem 3.6, where $\mathcal O_{v}$ is our
${\go}_v$).
By Artin reciprocity law, a necessary condition for
$\X(\go_S)\neq \emptyset$ is
\begin{equation}\label{bm}
\X(\A_{F,S})^{\Br(X_F)} \neq \emptyset
\end{equation}
When $\Br(X_F)$ is constant, the above condition is equivalent to
the classical local-global (Hasse) principle.

\begin{definition} If the condition (\ref{bm}) is also sufficient to insure that
$\X(\go_S)\neq \emptyset$, we say \emph{the Brauer-Manin obstruction
is the only obstruction} for {the} existence of integral points
{on} $\X$.
\end{definition}

We will use the following gluing process several times.

\begin{lem}\label{gluing} Let $v_1,\dots, v_n\in \Spec(\go_S)$ and
$V=\Spec\go_S\setminus \{v_1,\dots, v_n\}$.
Suppose we are given separated schemes of finite type $\Y$ over $V$,
$\Y_i$ over $\go_{F, v_i}$ and isomorphisms
$f_i : \Y_F\simeq (\Y_i)_{F}$. Then there exists
a unique (up to isomorphisms) separated scheme of finite type $\X$ over
$\go_S$ such that $\X_V\simeq \Y$ and $\X_{\go_{F,v_i}}\simeq
\Y_i$.

{If $\Y$, $\Y_1,\dots, \Y_n$ are group schemes and if
the $f_i$ are isomorphisms of group schemes, then $\X$ is a group
scheme over $\go_S$.}
\end{lem}

\begin{proof} Let $f_{ij}=f_j\circ f_i^{-1} : (\Y_i)_{F}\to (\Y_j)_{F}$.
Then $f_{jk}\circ f_{ij}=f_{ik}$. There are open neighborhoods
$V_i\ni v_i$ such that $\Y_i$ and $f_i$ extend to $V_i$ and
the identity $f_{jk}\circ f_{ij}=f_{ik}$ holds on $V_i\cap V_j\cap V_k$.
We can then use the usual gluing process to construct $\X$. The
separatedness of $\X$ comes from that of $\Y\to V$ and of
$\Y_i\to\Spec\go_{F, v_i}$. {The assertion on the
  structure of group scheme is straightforward.}
\end{proof}

\begin{rem} Let $\Y$ be an $\go_S$-scheme of finite type, not necessarily
flat, with generic fiber isomorphic to $X$. Let $\X$ be the
biggest flat closed subscheme of $\Y$ (defined by the
sheaf of ideals of $\go_S$-torsion elements of $\mathcal O_\X$ or,
equivalently, $\X$ is the scheme-theoretic closure in $\Y$ of the
generic fiber of $\Y$). Then
for any flat $\go_S$-scheme $T$, the canonical map
$\Y(T)\to \X(T)$ is bijective. So for the questions we deal
with in this paper, it is harmless to restrict ourselves to flat models.
On the other hand, for a model $\X$ to have local sections in ${\go}_{v}$
for all $v\notin S$, it is necessary that $\X$ be faithfully flat over
$\go_S$.
\end{rem}

\begin{prop} Let $R$ be a Dedekind domain with field of fractions $K$.
Let $X_K$ be an algebraic variety over $K$. Then
$X_K$ has a model $\X$ over $R$. Moreover there exists
a such $\X$ which is
proper (resp. projective) over $R$ if $X_K$ is proper (resp. projective)
over $K$.
\end{prop}

\begin{proof} The case when $X_K$ is projective is immediate: embed
$X_K$ into a projective space $\mathbb P^N_K$ and take the scheme-theoretic
closure $\X$ of $X_K$ in $\mathbb P^N_R$. As $\X$ is proper and flat
over $R$, it is faithfully flat.

Now suppose $X_K$ is affine.
Note that if we extend directly to $R$ by scaling a system of equations
defining $X_K$ and take the biggest closed subscheme flat over $R$, we
just get a flat, but not necessarily faithfully flat scheme over $R$.
Let $Y_K\subseteq \mathbb P^N_K$ be a
projective variety containing $X_K$ as a dense open subvariety. Consider
the scheme-theoretic closure $\Y$ of $Y_K$ in $\mathbb P^N_R$ and
${\mathbf Z}$ the scheme-theoretic closure of $Y_K\setminus X_K$ in
$\Y$. Then $\X:=\Y\setminus {\mathbf Z}$ is flat, separated and finite
type over $R$, with generic fiber equal to $X_K$. Let us show $\X$ is
faithfully flat. As $\Y, {\mathbf Z}$
are flat and proper over $R$, they are faithfully flat and
$$\dim{\mathbf Z}_v=\dim{\mathbf Z}_K < \dim \Y_K=\dim \Y_v, \quad
\forall v\in\Spec R.$$
Therefore $\X_v=\Y_v\setminus {\mathbf Z}_v\ne \emptyset$ and $\X$ is
faithfully flat.

In the general case, as $X_K$ is separated and
of finite type over $K$, it extends to a separated scheme $\X_1$ of
finite type and faithful flat over some open subset $V\subseteq \Spec R$
(\cite{EGA}, IV.8.8.2(ii)).
Let $U_K$ be a non-empty affine open subset of $X_K$. By the affine
case, there exists a faithfully flat model $\mathbf{U}$ of $U_K$ over
$R$. Shrinking $V$ if necessary, we can suppose
$U_K\subseteq X_K$ extends to an open immersion
$\mathbf{U}_{V} \hookrightarrow (\X_1)_{V}$. Let
$\X$ be the gluing of $\X_1$ and $\mathbf U$ along $\mathbf{U}_{V}$.
Then $\X$ is a model of $X_K$ over $R$.

Suppose $X_K$ is proper. The model $\X$ we obtained above is not
proper in general. But by Nagata's embedding theorem
(\cite{Lut}, \cite{Conrad2}), $\X$
is an open subscheme of some proper scheme $\mathbf Z$ over $R$.
As $\X$ is flat, it is contained in the scheme-theoretic
closure $\mathbf Z'$ of $X_F$ in $\mathbf Z$. Now
$\mathbf Z'$ is a proper model of $\X$ (it is faithfully flat over
$R$ by properness).
\end{proof}

\begin{rem}\label{smoothening}
Suppose $X_F$ is smooth. To study the integral points
of $X_F$ over $\go_S$
{when $X_F$ has local points at every $v\notin S$}, we can restrict ourselves to
smooth models in the sense that for any integral model $\X$
{such that $\X({\go}_v)\ne\emptyset$ for all $v\notin
  S$},
there exists a quasi-projective morphism
$\tilde{\X}\to \X$ of models of $X_F$ such that $\tilde{\X}$ is
smooth and the canonical maps
$\tilde{\X}(\go_S)\to \X(\go_S)$ and $\tilde{\X}({\go}_v)\to
\X({\go}_v)$, $v\notin S$, are bijective.
Indeed, $\X$ is already smooth over some dense open
subset $V$ of $\Spec(\go_S)$. By the gluing process \ref{gluing}, we
are reduced to the local case. Then the result follows from the smoothening
procedure, see \cite{BLR}, Theorems 3.1/3: if $\X'\to \X$ is a
smoothening, then take $\tilde{\X}$ equal to the smooth locus of
$\X'$. Note that for any closed point
$v\in\Spec(\go_S)$, we have $\tilde{\X}({\go}_v)=
\X({\go}_v)$ because the smoothening commutes with base
changes of index $1$ (see \cite{BLR}, Corollary 3.6/6).
{In particular
$\tilde{\X}({\go}_v)\ne\emptyset$ and $\tilde{\X}$ is faithfully flat
over $\go_{S}$.}
As a counterpart,
if we start with a proper model $\X$, the smoothening will
produce a non-proper smooth model in general.

If $X_F$ is furthermore a smooth algebraic group and
the model $\X$ is a (separated) group scheme, then
$\tilde{\X}$ can be chosen to be a smooth group scheme. This is a consequence
of the existence of group smoothening as shown in \cite{BLR}, 7.1/5.
The existence over a global Dedekind base scheme follows from
the local case by the gluing process as above. We also have
$\tilde{\X}({\go}_v)=\X({\go}_v)$ for all closed points $v$
because the group smoothening is
obtained by successive dilatations and the latter operations commute with
the passage to the completion.
\end{rem}

\begin{lem} \label{section-lg}
Let $\X$ be an integral model of $X_F$ over $\go_S$. Then the
canonical commutative diagram of injective maps
$$
\begin{CD}
\X(\go_S) @>>> \X(\A_{F,S}) \\
@VVV  @VVV \\
X_F(F) @>>> \prod_{v \in \Omega_F}  X_F(F_v)
\end{CD}
$$
allows us to identify $\X(\go_S)$ with
$X_F(F)\cap \X(\A_{F,S})$.
\end{lem}

\begin{proof} Let $\sigma\in X_F(F)$ such that $\sigma\in \X({\go}_v)$
in $X_F(F_v)$ for all $v\not \in S$. Then $\sigma\in \X(\go_{F,v})$ for all
$v\not \in S$. Moreover $\sigma$ extends to a $U(v)\to \X$ for
some open neighborhood $U(v)$ of $v\in \Spec\go_S$ for all $v\not\in
S$. One concludes that $\sigma\in \X(\go_S)$ by gluing these morphisms.
\end{proof}

Let $B$ be an integral domain containing $\go_S$, with field of
fractions $L$.
Let $\X$ be an integral model of $X_F$ over $\go_S$. The base change
$\Spec B\to \Spec L$ induces canonically a map
$\X(B) \to X_F(L)$
(depending on the isomorphism $\X_F\simeq X_F$) which is injective
because $\X$ is separated over $\go_S$. For simplicity,  we will consider
$\X(B)$ as a subset of $X_F(L)$. By Lemma \ref{section-lg},
$\X(\go_S)$ can be identified to $X_F(F)\cap \X(\A_{F,S})$.
\medskip

Next we relate the analytic open subsets to integral points of
models.

\begin{prop} \label{integral-U}Let $R$ be a discrete valuation ring
with field of fractions $K$. Let $\hat{R}$ be its completion and
$\hat{K}=\mathrm{Frac}(\hat{R})$. Let $X_K$ be an algebraic variety
over $K$, let $q\in X_K(\hat{K})$ and let $U \subseteq X_K(\hat{K})$
be an open neighborhood of $q$ for the topology induced by that of $\hat{K}$.
Then there exists an integral model $\X$ of 
$X_K$ over $R$ such that $q\in \X(\hat{R})\subseteq U$.
\end{prop}

\begin{proof} First let us prove that if $\X_0$ is a model of $X_K$ over $R$ such that
$q\in \X_0(\hat{R})$, then there exists a finite sequence
of dilatations (\cite{BLR}, \S 3.2)
$$\X_d\to \X_{d-1} \to \cdots \to \X_0$$
such that $\X_d(\hat{R})\subseteq U$.
Let $x_0\in (\X_0)_k(k)$ be the specialization of $q$
({\it i.e.} the image of the closed point of $\Spec(\hat{R})$ by
$q : \Spec(\hat{R})\to \X_0$).
Consider $\X_1\to \X_0$ the dilatation of $\{x_0\}=\Spec k(x_0)$ on $\X_0$
(see \cite{BLR}, \S 3.2). By the universal property of the
dilatation, the image of the (injective) canonical map
$\X_1(\hat{R})\to \X_0(\hat{R})$ consists of the points of
$X_K(\hat{K})$ whose specializations in $(\X_0)_k(k)$ are equal to $x_0$.
In particular, $q\in \X_1(\hat{R})$.
We define inductively a sequence $\X_{i+1}\to \X_i$ of dilatation of
$x_i$ on $\X_i$, where $x_i\in (\X_i)_k(k)$ is the specialization
of $q\in \X_i(\hat{R})$.
Let us show that $\X_d(\hat{R})\subseteq U$ when $d$ is big enough.
As the dilatation commutes
with flat base change (\cite{BLR}, 3.2/2(b)), we can suppose $R$ is
complete.

By construction, for any closed subscheme $Z_i$ of $(\X_i)_k$ and
for any open subscheme $\mathbf W$ of $\X_i$,
the dilatation of $Z_i\cap {\mathbf W}$ on $\mathbf W$
is the restriction to $\mathbf W$ of the dilatation $\X_{i+1}\to \X_i$.
So we can suppose $\X_0$, and hence $X_K$,
are affine. We have a closed immersion 
$\iota \colon \X_0\hookrightarrow {\mathbf A}_0:=\Spec R[x_1,\dots, x_n]$
such that $\iota(q)$ is the origin $(0, \dots, 0)$.
Let $t$ be a uniformizing element of $R$.
For any integer $i\ge 0$, let $D(0, i)=t^iR^n\subset {\mathbf A}_0(R)$.
Then by the definition of the topology on $X_K(K)$, there exists
$d\ge 0$ such that $\iota_K^{-1}(D(0, d))\subseteq U$.

By the construction of the dilatation (\cite{BLR}, 3.2/2(b)), $\iota$
induces a closed immersion
$$\iota_d \colon \X_d\hookrightarrow
{\mathbf A}_d:=\Spec R[x_1/t^d, \dots, x_n/t^d]$$ and
$$\X_d(R)\subseteq \iota_d^{-1}({\mathbf A}_d(R)) \cap X_K(K)
=\iota_K^{-1}(D(0, d))\cap X_K(K) \subseteq U.$$

Now to prove the proposition, it is enough to find a model $\X_0$ of $X_K$ over $R$ such that
$q\in \X_0(\hat{R})$. Let $W_K\subseteq X_K$ be an affine Zariski open neighborhood
of $q\in X_K(\hat{K})$. Then $W_K$ extends to an affine model
$\mathbf W$. Up to scaling a system of coordinates on $W_K$,
we can suppose that $q\in \mathbf W(\hat{R})$. Now we can glue $X_K$ and
${\mathbf W}$ along $W_K$ to get an integral model $\X_0$ of $X_K$
over $R$ with $q\in \X_0(\hat{R})$.
\end{proof}


\begin{cor}\label{U-model} Let $X_F$ be an algebraic variety over $F$.
Let $U \subseteq X_F(\mathbb A_F^S)$ be an open subset and let $z\in
U$. Then there exists an integral model $\X$ of
$X_F$ over $\go_S$ such that
$$z\in \prod_{v \notin S}\X({\go}_v)\subseteq U. $$
In other words, the subsets $ \prod_{v \notin S}\X({\go}_v) \subset X_F(\A_F^S)$,
where $\X$ runs through the integral models of $X_F$ over $\go_S$,
form a basis of the topological space $X_F(\A_F^S)$.

Moreover, if $X_F$ is a (smooth) algebraic group over $F$, then 
the subsets of the form $\prod_{v\notin S}\X(\go_v)$, where
$\X$ is a (smooth) group scheme model of $X_F$ over $\go_S$, form
a fundamental system of neighborhood of $1$.
\end{cor}

\begin{proof} {By the definition of the adelic topology
  (\cite{Conrad}, \S 4), there exists an integral model $\Y$ of $X_F$
over $\go_T$ for some finite subset $T\supseteq S$ of $\Omega_F$
such that
$$U\supseteq \prod_{v\in T\setminus S} U_v \times \prod_{v\notin T} \Y(\go_v)$$
for some non-empty open subsets $U_v$ of $X_F(F_v)$.}

Write $z=(z_v)_{v\not\in S}$. For any $v\in T\setminus S$,
there exists an integral model $\X_v$ of $X_F$ over $\go_{F,v}$ such that
$z_v\in {\X}_v({\go}_v)\subseteq U_v$ by Proposition~\ref{integral-U}(1).
Using Lemma~\ref{gluing}, we obtain an integral model $\X$ of
$X_F$ over $\go_S$ as required by gluing the schemes $\X_{v}$, $v\in T\setminus S$ with $\Y\times_{\go_S} \go_T$. 

The case of algebraic groups is proved in a similar 
way using dilatations along the unit section, see Remark~\ref{smoothening}.
\end{proof}

\subsection{Applications to strong approximations} 

Now we can give a Diophantine interpretation of strong approximation.

\begin{thm} \label{Diophantine-sap} Let $X_F$ be an algebraic variety
over $F$ and let $S$ be a non-empty finite subset of $\Omega_F$
containing $\infty_F$. Then
$X_F$ satisfies SAP-BM off $S$ if and only if the Brauer-Manin obstruction is
the only obstruction for the existence of integral points
for every integral model $\X$ of $X_F$ over $\go_S$.
\end{thm}

\begin{proof} ($\Rightarrow$) Let $\X$ be an integral model of
$X_F$ over $\go_S$ and suppose that
$$\X(\A_{F,S})^{\Br(X_F)} \neq \emptyset.$$
This set being open (with respect to the adelic topology) in $X_F(\mathbb A_F)^{\Br(X_F)}$, we have
$$\X(\go_S)=X_F(F) \cap P_S[\X(\A_{F,S})^{\Br(X_F)}]  \neq \emptyset $$
by hypothesis on $X_F$. Hence the Brauer-Manin obstruction is
the only obstruction for the existence of integral points
on $\X$.
\medskip

($\Leftarrow$) Let $U$ be an open subset of $X_F(\mathbb A^S_{F})$ such that
$$U \cap P_S(X_F(\mathbb A_F)^{\Br(X_{F})}) \neq \emptyset.$$
Let $z$ be a point of this set.
By Corollary~\ref{U-model}, there exists an integral model
$\X$ over $\go_S$ such that
$$z\in \prod_{v \not \in S} \X({\go}_v)\subseteq U . $$ This
implies that
$$\X(\A_{F,S})^{\Br(X_F)}\ne\emptyset.$$
By the hypothesis, one obtains that $\X(\go_S)\ne \emptyset$. As the
natural map $\X(\go_S)\to \X(\mathbb A_F^S)$ has image in
$\prod_{v \not \in S} \X({\go}_v)\subseteq U$ and
$\X(\go_S)\subseteq X_F(F)$, one concludes that $U$ meets the image of $X_F(F)$
in $X(\mathbb A_F^S)$ and the theorem is proved.
\end{proof}

\begin{cor}\label{bijective} VSAP-BM off $S$ holds for $X_F$
if and only if the diagonal map
$$\X(\go_S)\to P_S[\X(\A_{F,S})^{\Br(X_F)}] $$
is bijective for every model $\X$ of $X_F$ over $\go_S$.
\end{cor}

\begin{proof} ($\Rightarrow$) We only have to show that the diagonal map is surjective. Let
$$x\in\X(\A_{F,S})^{\Br(X_F)}\subseteq X_F(\A_F)^{\Br(X_F)}. $$
Then there exists $\sigma\in X_F(F)$ such that for all $v\notin S$,
$$\sigma_v=x_v\in \X({\go}_v)\cap X_F(F)=\X(\go_{F,v}),$$
where $\go_{F,v}$ is the local ring of $F$ at $v$. 
Hence $\sigma\in \X(\A_{F,S})\cap X_F(F)=\X(\go_S)$
and $\sigma=P_S(x)$ under the diagonal map.

($\Leftarrow$) This follows from Corollary \ref{U-model}.
\end{proof}

\section{Very strong approximation with Brauer-Manin obstruction}

A smooth curve $X_F$ is called hyperbolic if it is the complement of a
reduced separable effective divisor of degree $d\geq 0$ {on} a projective curve
of genus $g$ such that $2g-2+d>0$. By Siegel's theorem for integral
points and Faltings' theorem on Mordell's conjecture,
$\X(\go_S)$ is finite for any integral model $\X$ of $X_F$ over $\go_S$.
If $X_F$ satisfies SAP-BM off $S$ (Definition~\ref{def-BM}), then the diagonal map
$$ \X(\go_S) \to P_S[\X(\A_{F,S})^{\Br(X_F)}] $$
is bijective for any integral model $\X$ of $X_F$ over $\go_S$. By
Corollary~\ref{bijective},  $X_F$ satisfies the VSAP-BM off $S$. However, Harari
and Voloch gave an explicit example  at the end of \cite{HaVo} to
explain that the strong approximation property with Brauer-Manin obstruction
does not hold in general.

Let us first prove some general properties regarding the very strong
approximation property. Recall that an immersion is a closed immersion followed by an open immersion
(\cite{EGA}, I.4.2.1).

\begin{prop}\label{BM-VSAP-im}
Let $f: X_F\to Y_F$ be an immersion of algebraic varieties over a global
field $F$. Suppose that VSAP-BM off $S$ holds for $Y_F$,
then the same holds for $X_F$. \end{prop}

\begin{proof} We have the following commutative diagram
$$
\begin{CD}
X_F(F)  @>f>> Y_F(F) \\
@VVV  @VVV \\
X_F(\mathbb A_F)^{\Br(X_F)}  @>>> Y_F(\mathbb A_F)^{\Br(Y_F)} \subseteq Y_F(\mathbb A_F) \\
 @VP_SVV @VVP_SV  \\
 P_S[X_F(\mathbb A_F)^{\Br(X_F)}]  @>>> P_S [Y_F(\mathbb A_F)^{\Br(Y_F)}]
\subseteq Y_F(\mathbb A_F^S) \\
\end{CD}
$$
where the horizontal arrows are injective (the last two ones are injective
because $X_F(\A_F)\to Y_F(\A_F)$ and $X_F(\A_F^S)\to Y_F(\A_F^S)$ are injective).
Here we use the hypothesis that $f$ is an immersion (or at least
universally injective).
Moreover, if we consider $X_F(F), Y_F(F)$ and $X_F(\A_F^S)$ as subsets of $Y_F(\A_F^S)$, then
$Y_F(F) \cap X_F(\A_F^S)=X_F(F)$. Indeed, let
$y \in Y_F(F)$ be such that $(y)_v=x_v$ in $Y_F(F_v)$ with
$x_v\in X_F(F_v)$ for all $v\notin S$. Fix a $v_0\notin S$, then
$(y)_{v_0}\in X_F(F_{v_0})\cap Y_F(F)=X_F(F)$.
So $y\in X_F(F)$.
In a more algebraic setting,
we are saying that some $F$-algebra homomorphism
$\cO_{Y_F, y}\to F$ fits into a commutative diagram
$$
\begin{CD}
\cO_{Y_F,y} @>y>> F \\
@VVV @VVV \\
\cO_{X_F,x_{v_0}} @>x_{v_0}>> F_v
\end{CD}
$$
where the left vertical arrow is surjective (here we use the hypothesis
$f$ is an immersion). So the bottom arrow
takes values in $F$ and $x_{v_0}\in X_F(F)$. This implies immediately
the proposition.
\end{proof}

\begin{prop} \label{BM-VSAP-prod}
Let $X_F, Y_F$ be algebraic varieties over a global field $F$
such that $X_F(F), Y_F(F)\ne\emptyset$.
Then $X_F\times_F Y_F$ satisfies SAP-BM (resp. VSAP-BM) off S if and
only if the same holds for both $X_F$ and $Y_F$.
\end{prop}

\begin{proof} By functoriality, the projections to $X_F$ and $Y_F$
induce a continuous map
$$ (X_F\times_F  Y_F)(\A_F)^{\Br(X_F\times_F Y_F)} \to
X_F(\A_F)^{\Br(X_F)}\times Y_F(\A_F)^{\Br(Y_F)}$$
which is injective as both sides are subsets of $X_F(\A_F)\times Y_F(\A_F)$.
They both contain $(X_F\times_F Y_F)(F)$. The same holds when
we project to $X_F(\A_F^S)\times Y_F(\A_F^S)$. So if $X_F$ and $Y_F$ satisfy
VSAP-BM off $S$, then so does $X_F\times_F Y_F$. Suppose $X_F, Y_F$ satisfy
SAP-BM off $S$.
As $(X_F\times_F Y_F)(\A_F)$ is canonically homeomorphic to the
product $X_F(\A_F)\times Y_F(\A_F)$,
any non-empty open subset of $P_S[(X_F\times_F Y_F)(\A_F)^{\Br(X_F\times_F Y_F)}]$
is the trace of some $U\times V$ with $U, V$ respective open subsets
of $X_F(\A_F^S)$ and $Y_F(\A_F^S)$ such that $U\cap P_S(X_F(\A)^{\Br(X_F)})$
and $V\cap P_S(Y_F(\A)^{\Br(Y_F)})$ are non-empty. Thus we find rational points
$x\in X_F(F)\cap U$ and $y\in Y_F(F)\cap V$, therefore
$(x,y)\in (X_F\times_F Y_F)(F)\cap (U\times V)$ and $X_F\times_F Y_F$ satisfies
SAP-BM off $S$.

Now we notice that a rational point $y_0\in Y_F(F)$ gives rise
to a closed immersion
$$s_{y_0}: X_F\to X_F\times_F Y_F, \quad x\mapsto (x, y_0)$$
by base change. Thus $X_F$ (and similarly, $Y_F$) satisfies VSAP-BM
off $S$
if $X_F\times_F Y_F$ does (Proposition~\ref{BM-VSAP-im}). Suppose now
that $X_F\times_F Y_F$ satisfies SAP-BM off $S$. Let $U$ be a non-empty subset
of $P_S[X_F(\A_F)^{\Br(X_F)}]$. The intersection
$$(U\times P_S[Y_F(\A_F)^{\Br(Y_F)})]\cap P_S[(X_F\times_F Y_F)(\A_F)^{\Br(X_F\times_F Y_F)}]$$
is non-empty as it contains $U\times \{ y_0\}$,
thus contains a point $(x,y)\in (X_F\times_F Y_F)(F)$. By construction,
$x\in U\cap X_F(F)$, and $X_F$ satisfies SAP-BM
off $S$.
\end{proof}

Note that for projective smooth and geometrically integral varieties,
Skorobogatov and Zarhin proved a stronger result
(\cite{SZ}, Theorem C) than Proposition \ref{BM-VSAP-prod}.
\medskip

The simplest class for which the very strong approximation property with
Brauer-Manin obstruction holds are zero-dimensional schemes over $F$.

\begin{prop} Let $\Xi$ be the set of the complex places
    of $F$ if $F$ is a number field and $\Xi=\emptyset$ otherwise.
If $X_F=X_1\coprod X_2$ is a disjoint union of two open subschemes, then
$$P_{\Xi}[X_F(\mathbb A_F)^{\Br(X)}]=P_{\Xi}[X_1(\mathbb
A_F)^{\Br(X_1)}]\coprod P_{\Xi}[X_2(\mathbb A_F)^{\Br(X_2)}].$$
In particular, if $X_F$ is finite, then
$X_F$ satisfies the very strong approximation property with Brauer-Manin obstruction off any finite set containing $\Xi$.
\end{prop}

\begin{proof} If there is  $$(x_v)_{v\in (\Omega_F \setminus \Xi)}\in
  P_{\Xi}[X_F(\mathbb A_F)^{\Br(X)}]$$
with $v_1, v_2\in \Omega_F \setminus \Xi$ such that $x_{v_1}\in X_1(F_{v_1})$ and $x_{v_2}\in X_2(F_{v_2})$, then one can choose $\xi_1\in \Br(F)$ and $\xi_2\in \Br(F)$ satisfying
$$ \begin{cases} \inv_{v}(\xi_1) \neq \inv_v (\xi_2) \ \ \ & \text{if $v=v_1$ or $v_2$} \\
\inv_v(\xi_1)=\inv_v(\xi_2)=0 \ \ \ & \text{otherwise} \end{cases}$$ by the Artin reciprocity law
$$ 0 \rightarrow \Br(F) \rightarrow \oplus_{v\in \Omega_F} \Br(F_v) \xrightarrow{\sum_{v} inv_v} \Q/\Z \rightarrow 0 $$ (see \cite{NSW2}, VIII, (8.1.17) ). 

Let $\xi_i^*$ be the images of $\xi_i$ under
the canonical map $\Br(F)\rightarrow \Br(X_i)$ for $i=1, 2$. Then
$$(\xi_1^*, \xi_2^*) \in \Br(X_1)\oplus \Br(X_2) \cong \Br(X)$$
and
$$0= \sum_{v\in (\Omega_F\setminus \Xi)}
\inv_v((\xi_1^*+\xi_2^*)(x_v))=\inv_{v_1}(\xi_1)+\inv_{v_2}(\xi_2)=\inv_{v_1}(\xi_1)-\inv_{v_1}(\xi_2)$$ by the Artin reciprocity law as above.  A contradiction is derived and the first part of the proposition is proved.

Now suppose that $X_F$ is finite. By the above result,
we can suppose $X_F$ is local. By \cite{GB1}, Cor. 6.2, we can further suppose that
$X_F$ is reduced, hence is the spectrum of some finite field extension
$L$ of $F$. By Chebotarev Density Theorem,
if $X_F(\A_F)\ne\emptyset$, then $L=F$, in which case $X_F(F)=X_F(\A_F)$
and the result follows.
\end{proof}

Note that a similar decomposition result can be found in \cite{St},
Proposition~5.1.1.
\medskip

Let $T_F$ be a torus over a global field $F$. Let
$\widehat{T}_{F}:=\Hom_{\bar F}(T_{\bar F},\G_{m,\bar{F}})$ be the
group of characters of $T_F$. A point $x\in T_F(\A_F)$
induces an evaluation homomorphism
$$ \widehat{T}_F \xrightarrow{x}
{\mathbb I}_{\bar F}:=\varinjlim_{E/F} {\mathbb I}_E $$ where $E/F$ runs
over all finite separable subextensions of $\bar{F}$, thus a group homomorphism
$$ H^2(F, \widehat{T}_F) \xrightarrow{x^*} H^2(F, {\mathbb I}_{\bar F}) = \bigoplus_{v\in \Omega_F} \Br(F_v) \xrightarrow{\sum_v \inv_v} \Q/\Z $$
(\cite{NSW2}, VIII, (8.1.7)). This gives a pairing
\begin{equation}
  \label{eq:BM-pairing}
  T_F(\A_F)\times H^2(F, \widehat{T}_F)  \longrightarrow \Q/\Z
\end{equation}
and induces a homomorphism
\begin{equation}
  \label{eq:pairing-map}
  T_F(\A_F) \longrightarrow \Hom (H^2(F,\widehat{T}_F), \Q/\Z).
\end{equation}
This pairing is compatible with the Brauer-Manin pairing. Indeed, one has the canonical inclusion
 $$\Br(F)\oplus H^2(F, \widehat{T}_F)=
H^2(F, \cO(T_{\bar{F}})^\times) \hookrightarrow \Br(T_F) $$ by Hochschild-Serre
 spectral sequence (see \cite{Milne80}, III, Theorem 2.20)
with $\Pic(T_{\bar F})=0$. This inclusion commutes with the evaluation
 $$
\begin{CD}
H^2(F, \widehat{T}_F)  @>>> \Br(T_F)  \\
@V{x^*}VV   @VV{x^*}V \\
\bigoplus_{v\in \Omega_F} \Br(F_v) @{=} \bigoplus_{v\in \Omega_F} \Br(F_v)\\
\end{CD}
$$
for any $x\in T_F(\A_F)$.
\medskip

For any topological group $G$, denote by $G^{\wedge}$ the completion of $G$ with respect to open finite index subgroups. 

\begin{prop}\label{tori} Let $T_F$ be a torus over a global function
  field $F$. Then we have canonical exact sequences 
\begin{equation} \label{eq:profi}
1 \rightarrow T_F(F)^{\wedge}  \to  T_F(\A_F)^{\wedge}\rightarrow H^2(F, \widehat{T}_F)^{\vee}
\end{equation} 
and
  \begin{equation}
    \label{eq:H2-fct}
    1\rightarrow T_F(F) \rightarrow T_F(\A_F) \rightarrow H^2(F, \widehat{T}_F)^{\vee}
  \end{equation}
where $^{\vee}$ is the Pontryagin dual. 
\end{prop}

\begin{proof}   The first exact sequence follows from Theorem 3.11 in \cite{GT}  by taking $M=T$. 
Indeed, the hyper-cohomology is the same as the usual Galois cohomology in this case and $M^*=\widehat{T}[1]$. Since $H^3(F, \widehat{T})$ is torsion, one has $H^2(F, M^*)=H^3(F, \widehat{T})=0$ by VIII, (8.3.17) in \cite{NSW2}.    

 The second one is deduced from \eqref{eq:profi}
if we can prove that the canonical map 
$$T_F(\A_F)/T_F(F)\to T_F(\A_F)^{\wedge}/T_F(F)^{\wedge}$$ is injective.  Since $T_F(F)$ is discrete in 
$T_F(\A_F)$,  the quotient group $T_F(\A_F)/T_F(F)$ is locally compact and totally disconnected.  Let $S$ be a finite non-empty subset of $\Omega_F$ and ${\bf T}$ be a group scheme model of $T$ over $\go_S$. Then  
 $T_F(\A_F)/T_F(F){\bf T}(\A_{F,S})$ is finite (when $T_F$ is split this comes from the finiteness of $\Pic(\go_S)$,
the general case follows from \cite{Harder}, Lemma 2.2.3).  It is clear that $T_F(F){\bf T}(\A_{F,S})/T_F(F)$ is compactly generated. This implies that $T_F(\A_F)/T_F(F)$ is compactly generated. Hence  
$$T_F(\A_F)/T_F(F)\hookrightarrow (T_F(\A_F)/T_F(F))^{\wedge}$$ 
(see {\it e.g.} \cite{HS}, Appendix).  Since the natural projection  $T_F(\A_F) \rightarrow T_F(\A_F)/T_F(F)$ is open, one concludes 
$$(T_F(\A_F)/T_F(F))^{\wedge}=T_F(\A_F)^{\wedge}/T_F(F)^{\wedge} $$ by  Appendix in \cite{HS}. 
\end{proof}

\begin{thm}\label{stp}
Let $T_F$ be a torus over a global field $F$.
Suppose that one of the following conditions holds:
\begin{enumerate}
\item $F$ is a global function field;\footnote{In this case $\infty_F=\emptyset$ 
and $S$ can be empty, $P_{\emptyset}$ is then the identity map};
\item $F$ is a number field and
$$\rank_F(T_F)=\sum_{v\in \infty_F} \rank_{F_v}(T_{F_v})$$
where $\rank_F(T_F)$ and $\rank_{F_v}(T_{F_v})$ are the respective ranks of the free abelian groups $\Hom_F(T_F,\mathbb G_m)$ and  $\Hom_{F_v}(T_{F_v},\mathbb G_m)$.
\end{enumerate}
Then $T_F$ satisfies VSAP-BM off $S$ for any finite subset $S$ containing $\infty_F$. 
In Case (2), $T_F(F)$ is discrete, hence closed in $T_F(\A_\Q^{\infty_F})$. 
\end{thm}

\begin{proof} It is enough to show that, under the hypothesis of the
theorem, the kernel of the map \eqref{eq:pairing-map} is $T_F(F)$.

(1) The case where $F$ is a global function field follows from Proposition~\ref{tori}.

(2) Suppose $F$ is a number field. According to \cite{Shyr}, Main Theorem,
under the hypothesis on the ranks, there is an open subgroup of
$T(\mathbb A_F)$ whose intersection with
$T_F(F)$ is finite.
So $T_F(F)$ is discrete (hence closed) in
$T_F(\A_F^{\infty_F})$. The result follows from \cite{H}, Th\'eor\`eme 2 (see also Remarques that follow).
\end{proof}

Recall that a subvariety is a closed subvariety of an open subvariety. 

\begin{cor}\label{reb} Let $X_F$ be a subvariety of a torus as
in Theorem \ref{stp}. Then
$$X_F(F)= P_{\infty_F} (X_F(\A_F)^{\Br(X)})$$
where $\infty_F$ is empty if $F$ is a global function field.
In particular, for any finite $S\supseteq \infty_F$,
the very strong approximation property with Brauer-Manin
obstruction off $S$ holds for $X_F$.
\end{cor}

\begin{proof} This follows from Proposition~\ref{BM-VSAP-im} and
Theorem~\ref{stp}.
\end{proof}

\begin{example}\label{rank0} The following tori satisfy the rank condition in
Theorem \ref{stp} (2):
\begin{enumerate}
\item  when $F=\Q$ or is a quadratic imaginary field and $T_F$ is a
split torus over $F$;
\item when $F$ is a totally real field, and
$T_F=\prod_{i=1}^d\Res_{L_i/F}^1(\G_{m, L_i})$ where the $L_i$ are totally
imaginary quadratic extensions of $F$ and $\Res_{L_i/F}$ denotes the Weil 
restriction functor;
\item when $T_F$ is isogeneous to any finite product of the above tori
  over $F$.
\end{enumerate}
\end{example}

\begin{example}\label{ex-3.9} Let $X_{\Q}$ be one of the algebraic varieties
defined as follows. Then VSAP-BM off $\infty_{\Q}$ holds for $X_{\Q}$.
\begin{enumerate}
\item Let $X_{\Q}$ be defined by the following diagonal equation
$$ \sum_{i=1}^d a_i x_i^{n_i}=a_0 \ \ \ \text{with} \ \ \
\prod_{i=1}^d x_i \neq 0 $$ where $a_i\in\Q$,  $n_i\in \N\setminus \{0\}$
for $1\leq i\leq d$.
Then $X_{\Q}$ is a closed subvariety of $\G_{m, {\Q}}^d$.
\item Let $E$ be an elliptic curve over $\Q$ and
$z_1, z_2, z_3\in E(\Q)[3]$ be three distinct points. Then
$$X_{\Q}=E\setminus \{z_1, z_2, z_3\}$$
can be embedded in $\G_{m,\Q}^2$ because it is contained in the
complementary in $\mathbf P^2_\Q$ of the tangent lines at the three missing
points.
\item Any open subset of $\mathbf P^1_\Q$ in the complementary of an
imaginary quadratic point is contained in a torus as in
Example~\ref{rank0}(2).
  \end{enumerate}
\end{example}

\begin{rem}\label{open} Suppose $X_F$ is a smooth and geometrically integral variety over $F$ and $U$ is an open subscheme of $X_F$. If $\Br(U)/\Br(F)$ is finite, then
the strong approximation property with Brauer-Manin obstruction of $U$
implies the strong approximation property with Brauer-Manin obstruction of $X_F$ by Proposition 2.6 in \cite{CTX11}. It is natural to ask if the finiteness condition for $\Br(U)/\Br(F)$ can be removed. \end{rem}

Take any $X_F$ for which the strong approximation property with Brauer-Manin
obstruction does not hold ({\it e.g.}, the complement of a rational point
in some elliptic curve as at the end of \cite{HaVo}), then
the next corollary (or Example \ref{ex-3.9}(2)) gives a negative answer to the above question,
by taking a small enough open subset $U$ of $X_F$.

\begin{cor} \label{rev} Let $X_F$ be an algebraic variety over $F$, where $F=\Q$ or
is an imaginary quadratic field. Then there exists a dense open subset $U$ of
$X_F$ such that $U$ satisfies the very strong approximation property with
Brauer-Manin obstruction off $S$ for any finite subset $S$ containing $\infty_F$.  
\end{cor}

\begin{proof} By Corollary~\ref{reb} and Example~\ref{rank0}(1), it is enough to show that $X_F$ contains a dense open subset $U$
which is a finite disjoint union of subvarieties of split tori over $F$. After restricting to the interior of 
each irreducible component, we can suppose that $X_F$ is irreducible and
even integral. Shrinking further $X_F$ if necessary, we can suppose that $X_F$ is affine. 
Now embed $X_F$ as a closed subvariety
in some affine space $\mathbb A^n_F$ with $n$ minimal. Let $T$ be a split torus obtained as the
complement of $n$ independent linear hyperplanes in $\mathbb A^n_F$. Then $U:=X_F\cap T$ is a closed subvariety of $T$,
is non-empty (by the minimality of $n$) and is open in $X_F$.
\end{proof}

\begin{example} Let $F$ be a number field.
Let $X_F$ be an algebraic variety admitting
a non-constant morphism $f : \mathbf A^1_F\to X_F$ from the affine
line over $F$. Then for any finite $S\supseteq \infty_F$, the map
$$X_F(F)\to P_S[X_F(\mathbb A_F)^{\Br(X_F)}]$$
is not surjective. This can be proved as follows.
The image of $f$ in $X_F$ is a locally closed
subset of dimension $1$. Endowed with the reduced structure, it
is an integral subvariety of $X_F$. By Proposition~\ref{BM-VSAP-im},
we can replace $X_F$ by $f(\mathbf A^1_F)$ and
suppose that $f$ is surjective.

We can extend $f$ to a finite surjective morphism
$\pi : \mathbf P^1_F\to \overline{X}_F$ to a compactification
$\overline{X}_F$ of $X_F$. Then $\pi^{-1}(X_F)=\mathbf A^1_F$.
By L\"uroth's Theorem,
the normalization of $X_F$ is isomorphic to $\mathbf A^1_F$. Thus we
can replace $f$ by the normalization map of $X_F$ and
suppose that $f$ is birational.
Let $V$ be the smooth locus of $X_F$ and let $U=f^{-1}(V)$.
The canonical injective map
$$P_S[U(\A_F)]\cong P_S[V(\A_F)]\to P_S[X_F(\A_F)]$$ factors
through
$$P_S[U(\A_F)]\subseteq P_S[\mathbf
A^1_F(\A_F)]=P_S[\mathbf A^1_F(\A_F)^{\Br(\mathbf A^1_F)}]\to
P_S[X_F(\A_F)^{\Br(X_F)}].$$
So $P_S[U(\A_F)]$ is actually contained in
$P_S[X_F(\A_F)^{\Br(X_F)}]$. If the map $$X_F(F)\to P_S[X_F(\A_F)^{\Br(X_F)}]$$
is surjective, as $X_F(F)$ is countable,
$P_S[U(\A_F)]$ would also be countable, we have a contradiction.
\end{example}

\section{Harari-Voloch's conjecture}\label{hv}

Let $X_F$ be a hyperbolic rational curve ({\it i.e.} $X_F$ is an open subset
in $\mathbf P^1_F$, and its complementary is a reduced separable
divisor of degree $\ge 3$) and let $\X$ be an integral
model of $X_F$ over $\go_S$ for a non-empty finite subset $S$ of $\Omega_F$
containing $\infty_F$. Harari and Voloch proposed the following conjecture
(\cite{HaVo}, Conjecture 2).

\begin{conjecture}{(Harari-Voloch)} \label{conjecture}
Let $$B_S(X_F):= \Ker[\Br(X_F) \rightarrow \prod_{v\in S}
\Br(X_{F_v})/\Br(F_v)].$$
Then the diagonal map
$$ \X (\go_S) \rightarrow (\prod_{v\not\in S}\X({\go}_v))^{B_S(X_F)} $$
is bijective.
\end{conjecture}

\begin{rem} \label{HV-equiv}
By the same argument as those in Corollary \ref{bijective}, this conjecture
holds for all models of $X_F$ over $\Spec(\go_S)$ if and only if 
$$ X_F(F)\to X_F({\mathbb A}_F^S)^{B_S(X_F)}$$
is bijective. Since
$$X_F({\mathbb A}_F^S)^{B_S(X_F)}\supseteq P_S
[(X_F(\A_{F}))^{\Br(X_F)}]\supseteq X_F(F), $$
Conjecture~\ref{conjecture} implies the very strong approximation
property with Brauer-Manin obstruction off $S$ for $X_F$, and it
also implies a conjecture of Skolem (see Remark 2.5 in \cite{HaVo}).
\end{rem}

\begin{lem}\label{br} Let $F$ be any field.
Let $X_F$ be an affine curve contained in
${\mathbf P}^1_F$ with $X_F(F)\ne\emptyset$.
Let $\bar{F}$ be a separable closure of $F$.
\begin{enumerate}
\item There is a split exact sequence
$$0\to \Br(F)\to \Br(X_F) \to H^2(F,
\cO(X_{\bar{F}})^\times/\bar{F}^\times)
\to 0.$$
\item Suppose $F$ is a global field. Denote by
$\bar{F}_v$ a separable closure of $F_v$ for any $v\in S$.
Then one has a split exact sequence
\begin{multline}
0\to \Br(F)\to B_S(X_F)
\to \\
\Ker\left[H^2(F,\cO(X_{\bar{F}})^\times/\bar{F}^\times)\xrightarrow{\res}
  \prod_{v\in S}
  H^2(F_v,\cO(X_{\bar{F}_v})^\times/\bar{F}_v^\times)\right] \to 0
\end{multline}
for any finite subset $S$ of $\Omega_F$ containing $\infty_F$.
\end{enumerate}
\end{lem}

\begin{proof} By Hochschild-Serre spectral sequence (see
\cite{Milne80}, III, Theorem 2.20 and Appendix B, p. 309)
$$H^p(G, H^q(X_{\bar{F}}, {\mathbb G}_m)) \Rightarrow
H^{p+q}(X_F, {\G}_m)$$
with $G=\Gal(\bar{F}/F)$, one has the long exact sequence
$$(\Pic(X_{\bar{F}}))^G \to H^2(F, \cO(X_{\bar{F}})^\times)
\to  \Ker[\Br(X_F)\to \Br(X_{\bar{F}})]\to H^1(F, \Pic(X_{\bar{F}})).$$
Since $X_F$ is an affine open subscheme of ${\mathbf P}^1_F$, one has
$\Pic(X_{\bar{F}})=0$. Moreover, $\Br(X_{\bar{F}})\subseteq \Br(\bar{F}(X_{\bar{F}}))=0$ by Tsen's Theorem. Therefore
$$\Br(X_F)\cong H^2(F, \cO(X_{\bar{F}})^\times). $$
The existence of a $F$-rational point on $X_F$ implies that the exact
sequence of $G$-modules
$$1\to \bar{F}^\times \to \cO(X_{\bar{F}})^\times \to
\cO(X_{\bar{F}})^\times /\bar{F}^\times \to 1$$
splits. This give a split exact sequence
$$0\to \Br(F)\to \Br(X_F)\to H^2(F,
\cO(X_{\bar{F}})^\times/\bar{F}^\times)
\to 0.$$
When $F$ is a global field, the same decomposition holds over $F_v$
for any $v\in S$, and the lemma follows easily.
\end{proof}

Let $F$ be any field and $X_F$ be an affine curve contained in ${\mathbf P}^1_F$ such that
the reduced divisor $D:=\mathbf P^1_F \setminus X_F$ is separable of degree
$\ge 2$ and $X_F(F)\ne\emptyset$. The next lemma is the same as part
of \cite{HaVo}, Lemma 2.1, but we do not assume $H^3(F, \G_m)=0$.

\begin{lem} \label{embed-T} Let $X_F$ be as above and let $\bar{F}$ be a separable
  closure of $F$ with Galois group $G$. Let $\bar{D}=D\times_F \bar{F}$ and denote by
$\Div_{\bar{D}}({\mathbf P}_{\bar{F}}^1)$ the group of divisors on ${\mathbf P}^1_{\bar{F}}$ supported in $\bar{D}$.
Let $T_F$ be a torus over $F$ such that its character group $\widehat{T}_F$
is isomorphic to $\Div_{\bar{D}}^0 ({\mathbf P}^1_{\bar{F}})$ as a $G$-module.
Then there exists a closed immersion
$$ j : X_F\to T_F$$
such that the induced group homomorphism
$$\cO(T_{\bar{F}})^\times/\bar{F}^\times \to
\cO(X_{\bar{F}})^\times/\bar{F}^\times $$
is an isomorphism.
\end{lem}

\begin{proof}  We have a
canonical exact sequence
$$ 1\to \bar{F}^\times \to \cO(X_{\bar{F}})^\times \to \Div_{\bar{D}}({\mathbf P}^1_{\bar{F}})$$
where the last map consists in taking the divisors of rational
functions on $X_{\bar{F}}$. Because $\Pic^0(\mathbf P^1_{\bar{F}})=1$, it is clear
that the above exact sequence induces an exact sequence of $G$-modules
$$ 1\to \bar{F}^\times \to \cO(X_{\bar{F}})^\times \to \Div_{\bar{D}}^0 ({\mathbf P}_{\bar{F}}^1)\to 0.$$
A rational point on $X_F$ gives rise to an evaluation homomorphism
$e: \cO(X_{\bar{F}})^\times \to \bar{F}^\times$ of $G$-modules,
which is identity on $\bar{F}^\times$ on the source. Thus we have
$G$-equivariant isomorphisms
$$ \Div_{\bar{D}}^0 ({\mathbf P}^1_{\bar{F}})\cong \cO(X_{\bar{F}})^\times/\bar{F}^\times \cong
e^{-1}(\{1\}) \subset \cO(X_{\bar{F}})^\times,$$
and a $G$-equivariant $\bar{F}$-algebras homomorphism
$$ \bar{F}[\Div_{\bar{D}}^0 ({\mathbf P}^1_{\bar{F}})]\twoheadrightarrow
\bar{F}[e^{-1}(\{1\})]=\bar{F}[\cO(X_{\bar{F}})^\times]\subseteq \cO(X_{\bar{F}}).$$
An easy explicit computation (using the coordinates of the points of
$\bar{D}$; the hypothesis $\deg D\ge 2$ is needed here) shows that the last two terms are actually equal. Then
the above homomorphism induces a closed immersion $X_F\to T_F$
(depending on the choice of a rational point on $X_F$), and an
isomorphism
$$\cO(T_{\bar{F}})^\times/\bar{F}^\times =
\Div_{\bar{D}}^0 ({\mathbf P}^1_{\bar{F}})\to
\cO(X_{\bar{F}})^\times/\bar{F}^\times.$$ \end{proof}

\begin{definition}\label{B1S}
Let $T_F$ be a torus over $F$ and let
$\widehat{T}_{F}$ be the group of characters
of $T_{F}$. When $F$ is a global field, we set
$$B_{1,S}(T_F)=\Ker(H^2(F, \widehat{T}_{{F}})\to
\prod_{v\in S} H^2(F_v,\widehat{T}_{{F}})).$$
For any field $F$, we have
$\widehat{T}_{{F}}\cong \cO(T_{\bar{F}})^\times / \bar{F}^\times$
as Galois modules. Similarly to the proof of Lemma~\ref{br}(1), we
have a split exact sequence
$$0\to \Br(F) \to \Br_1(T_F):=\Ker(\Br(T_F)\to \Br(T_{\bar{F}}))\to H^2(F,
\widehat{T}_{{F}})
\to 0.$$
\end{definition}

The following proposition is a variant of Theorem 1 in \cite{HaVo}.

\begin{prop} \label{eqhv} Let $X_F$ be a hyperbolic rational curve
 over a global field $F$ and let
$j : X_F\to T_F$ be the closed immersion given by Lemma
\ref{embed-T}. Fix a non-empty finite set $S$ of primes containing $\infty_F$.
Then
\begin{enumerate}
\item we have
$$\overline{T_F(F)}^S =T_F(\A_F^S)^{B_{1,S}(T_F)}$$
where $\overline{T_F(F)}^S$ is the topological closure of $T_F(F)$ inside
$T_F(\A_F^S)$;
\item Conjecture \ref{conjecture} for all models $\X$ of $X_F$ over $\go_S$ is
true if and only if
$$ \overline{T_F(F)}^S \cap X_F(\A_F^S)=X_F(F).$$
\end{enumerate}
\end{prop}

\begin{proof} By
Lemmas \ref{br} and \ref{embed-T}, we have a split exact sequence:
\begin{equation}
  \label{eq:BS}
   0\to \Br(F)\to B_S(X_F)\to B_{1,S}(T_F)\to 0.
\end{equation}
By Theorem 2 in \cite{H} {when $F$ is a number field, and
by Proposition~\ref{tori} when $F$ is a function field,}
together with the compatibility of
local duality (\cite{NSW2}, VII, (7.2.9)), one has the following diagram of exact sequences
$$
\begin{CD}
 @. (\prod_{v\in \infty_F} \pi_0(T_F(F_v))) \times \prod_{v\in S\setminus \infty_F} T_F(F_v)  @>\phi>> \prod_{v\in S} H^2(F_v, \widehat{T}_{F})^{\vee}  \\
 @.  @VVV  @VV{\res^{\vee}}V \\
 \overline{T_F(F)} @>>>  (\prod_{v\in \infty_F} \pi_0(T_F(F_v))) \times T_F(\A_F^{\infty_F}) @>{\rho}>>  H^2(F, \widehat{T}_{F})^{\vee} \\
 @. @V{P_{0,S}}VV @VVV \\
 @. T_F(\A_F^S) @>>> B_{1,S}(T_F)^{\vee}
\end{CD}$$
where $^{\vee}$ is the Pontryagin dual,
$\overline{T_F(F)}$ is the topological closure of $T_F(F)$ inside
$(\prod_{v\in \infty_F} \pi_0(T_F(F_v))) \times T_F(\A_F^{\infty_F})$,
$\pi_0(T_F(F_v))$ is the group of connected components of the
topological group $T_F(F_v)$ {and $P_{0,S}$ is the
  projection map}. Moreover, $\phi$ is injective with dense
image by \cite{NSW2}, VII, (7.2.10). 
Write for simplicity this diagram as
$$
\begin{CD}
 @. T_S  @>\phi>> H_S^{\vee}  \\
 @.  @VVV  @VV{\res^{\vee}}V \\
 \overline{T_F(F)} @>>>  G @>{\rho}>>  H^{\vee} \\
 @. @V{P_{0,S}}VV @VVV \\
 @. T_F(\A_F^S) @>>> B_S^{\vee}
\end{CD}$$
and denote $T^S_0=T_F(\A_F^S)^{B_{1,S}(T_F)}$. 
As the latter is closed in $T_F(\A_F^S)$, to prove (1), one only needs to show that $T_F(F)T_S$ is dense in 
$P_{0,S}^{-1}(T^S_0)$ or, equivalently, that $\overline{T_F(F)}T_S$ is dense in 
$P_{0,S}^{-1}(T^S_0)$.

When $F$ is a number field, the image
  of $\rho$ is the kernel of a continuous map of topological groups
$H^2(F, \widehat{T}_{F})^{\vee}\to \Sha^1(T_F)$ by \cite{H},
Theorem 2, hence compact. Because its domain is $\sigma$-compact,
$\rho$ is an open map to its image. 
Let $U\subseteq G$ be an open subset such that 
$$ U\cap P_{0,S}^{-1}(T^S_0) \neq \emptyset. $$ 
Then $\rho(U)=W\cap \rho(G)$ for some open subset $W$ of $H^{\vee}$. As 
$({\res^{\vee}})^{-1}(W)$ is a non-empty open subset of $H_S^{\vee} $ by the hypothesis on $U$,
it contains an element of $\phi(T_S)$ by the density of $\phi(T_S)$. 
This implies that $U$ contains an element of $(\Ker\rho)T_S=\overline{T_F(F)}T_S$ and
proves (1) for number fields. 

Now suppose that $F$ is a global function field. Then $$\overline{T_F(F)}=T_F(F) \ \ \ \text{ and} \ \ \ 
G=T_F(\A_F) . $$ Let $U$ be an open subset of $T_F(\A_F)$ with an intersection point $g\in U\cap P_{0,S}^{-1}(T^S_0)$. 
Let $C$ be the integral smooth projective curve whose function field is $F$. 
Then $g^{-1}.U$ contains an open subgroup of the form 
$\prod_v \T(\go_v)$, where $\T$ is a group scheme of finite type over $C$
with generic fiber isomorphic to $T_F$ (Corollary~\ref{U-model}). We can shrink $U$ and suppose $U=g\prod_v \T(\go_v)$. As $T_F(F)T_S$ is a group, 
$U\cap T_F(F)T_S\ne \emptyset$ is equivalent to $1\in gH$, where 
$$H:=(\prod_v \T(\go_v))T_F(F)T_S=T_F(F)\Big(\prod_{v\in S}T_F(F_v)\times \prod_{v\notin S} \T(\go_v)\Big)$$ 
is an open subgroup of $T_F(\A_F)$. 
We know that $T_F(\A_F)/H$ is finite (when $T_F$ is split this comes from the finiteness of $\Pic(\go_S)$,
the general case follows from \cite{Harder}, Lemma 2.2.3). 
Therefore, $H$ is an open subgroup of finite index in $T_F(\A_F)$ and $gH^{\wedge}$ is an open
subset in $T_F(\A_F)^{\wedge}$. 
Similarly to the case of number fields, by using the exact sequence \eqref{eq:profi}, we see that 
$gH^{\wedge}\cap T_F(F)^{\wedge}T_S\ne\emptyset$. As $T_F(F)$ and $T_S$ are contained in $H$, 
$T_F(F)^{\wedge}T_S\subseteq H^{\wedge}$, hence $g\in H^{\wedge}\cap T_F(\A_F)=H$ and $1\in gH$.

(2) By (1), we have 
$$\overline{T_F(F)}^S \cap X_F(\mathbb  A_F^S)= (T_F(\A_F^S))^{B_{1,S}(T_F)} \cap X_F(\mathbb  A_F^S)=X_F(\mathbb  A_F^S)^{B_S(X_F)}$$
by the functoriality of the Brauer-Manin pairing (see \cite{Sk}, (5.3))
and by the exact sequence \eqref{eq:BS}. This completes
the proof of (2) by Remark~\ref{HV-equiv}.
\end{proof}

\begin{cor}\label{true} Conjecture \ref{conjecture} is true if $F$ is a global function field.
\end{cor}

\begin{proof} Embed $X_F$ into a torus $T_F$ as
in Lemma~\ref{embed-T}. By the above proposition, it is enough to show
that $\overline{T_F(F)}^S \cap X_F(\A_F^S)\subseteq X_F(F)$. For any $x\in
  \overline{T_F(F)}^S \cap X_F(\A_F^S)$, there is a finite set $S_1\supseteq S$ such that $T_F$ can be extended to
  a group scheme $\T$ of finite type over $\go_{S_1}$ with
  $$x\in \prod_{v\in S_1\setminus S}T_F(F_v) \times  \prod_{v\not\in S_1} {\T}(\go_v). $$
The latter is an open subset whose intersection with
$T_F(F)$ is equal to $\T(\go_{S_1})$.
This implies that $x\in \overline{{\T}(\go_{S_1})}$ where $\overline{{\T}(\go_{S_1})}$ is the topological closure of ${\T}(\go_{S_1})$ in $\prod_{v\not \in S} T_F(F_v)$ with the product topology. Since ${\T}(\go_{S_1})$ is a finitely generated abelian group, one has
$$x\in \overline{{\T}(\go_{S_1})} \cap X_F(\A_F^S)={\T}(\go_{S_1})\cap
X_F(F) \subseteq X_F(F)$$ by \cite{Sun}, Theorem 1. The proof is
complete. \end{proof}

\begin{example}\label{hv-ev} (See also Corollary~\ref{hv-ex2}) Let us give some evidence for Harari-Voloch conjecture over $\Q$ with $S=\{ \infty\}$.  
Let $X_\Q$ be a non-empty open subset of ${\mathbf P}^1_\Q \setminus D$, where 
$D$ consists of either three rational points, or two imaginary quadratic points, or 
one point of each type. Then Harari-Voloch conjecture is true for $X_\Q$ and $S=\{\infty\}$. 

Indeed, the torus $T_\Q$ which the curve $Y_\Q:={\mathbf P}^1_\Q\setminus D$ is embedded in following 
Lemma~\ref{embed-T} satisfies the condition of Theorem~\ref{stp} (2). In particular, 
$$\overline{T_\Q(\A_\Q)}^S=T_\Q(\Q). $$    It follows from Proposition~\ref{eqhv} (2) that Harari-Voloch
conjecture holds for $Y_\Q$ and $S$. For arbitrary non-empty open subset $X_\Q$ of $Y_\Q$, this
proof does not work anymore, but the result is still true by applying 
\cite{HaVo}, Theorem~3 to the open immersion $X_\Q\to Y_\Q$. 
\end{example}

\section{Adelic points of zero-dimensional subvarieties of tori}

In this section we will show that the analogue of \cite{St}, Theorem 3.11 
holds for tori over number fields. See Theorem~\ref{inters}. 
As an application, we generalize \cite{HaVo}, Theorem 3. See 
Proposition~\ref{quasi-finite}. 
Throughout this section, \emph{$F$ is a number field}. We also fix a finite subset $S$ 
of $\Omega_F$ containing $\infty_F$. 

The proof of the main theorem~\ref{inters} of this
section follows closely the same strategy as in \cite{St}. 
Let us first show some preliminary results on Galois cohomology of tori. Let $S_1$ be a finite subset of $\Omega_F$ with
$S_1\supseteq S$ and ${\T}$ be a separated commutative group scheme of finite type over $\go_{S_1}$ such that  
\begin{equation} \label{split}  {\bf T} \times_{\go_{S_1}} \bar{\go}_{S_1}  \cong \G_{m, \bar{\go}_{S_1}}^d \end{equation} 
for some $d\ge 1$, where  $\bar{\go}_{S_1}$ is the integral closure of $\go_{S_1}$ inside the algebraic closure $\bar{F}$ of $F$. Under the assumption (\ref{split}),  the multiplication by $N$ 
on ${\bf T}(\bar{\go}_{S_1})={(\bar{\go}_{S_1}^{\star})}^d$ is surjective. Note that 
enlarging $S_1$ if necessary, Condition~\eqref{split} is always be satisfied. Indeed, we can extend $F$ and suppose 
${\bf T}_{F}\simeq \G_{m, F}^d$. As ${\bf T}$ and $\G_{m, \go_{S_1}}^d$ coincide on the generic fiber
and are both of finite type, they concide on a dense open subset of $\Spec\go_{S_1}$.

For any positive integer $N$, denote by ${\T}[N]$ the group scheme of $N$-torsion of ${\T}$. In what follows, in Galois cohomology groups,
${\T}$ (resp. ${\T}[N]$) is the Galois module ${\T}(\bar{\go}_{S_1})$ (resp. ${\T}[N](\bar{\go}_{S_1})$). Define  as usual
$$ \Sel^{(N)}_S(F, {\T}):= \Ker(H^1(F, {\T}[N]) \rightarrow \prod_{v\not\in S}  H^1(F_v,{\T})) $$
and
$$\Sha_S({\T}):=\Ker(H^1(F, {\T}) \rightarrow \prod_{v\not\in S} H^1(F_v, {\T})).$$
(It can be seen that the image of
$H^1(F, {\T})$ in $\prod_{v\not\in S} H^1(F_v, {\T})$ is in fact
contained in $\oplus_{v\notin S} H^1(F_v, {\T})$ by Lang's Theorem on torsors under
connected algebraic groups over a finite field).
Then one has the short exact sequence
\begin{equation} \label{short} 0 \rightarrow {\T}(\go_{S_1})/N{\T}(\go_{S_1}) \rightarrow \Sel^{(N)}_S(F, {\T}) \rightarrow \Sha_S({\T})[N] \rightarrow 0 \end{equation}
and the natural coordinate map can be decomposed as
\begin{equation} \label{sel}
{\T}(\go_{S_1})/N{\T}(\go_{S_1}) \rightarrow
\Sel^{(N)}_S(F, {\T}) \rightarrow (\prod_{v\not\in S_1}  {\T}(\go_v)/N{\T}(\go_v))\times \prod_{v\in S_1\setminus S}  T_F(F_v)/NT_F(F_v)
\end{equation}
by using the Kummer sequence and Galois cohomology.
\begin{lem} \label{sha} The group $\Sha_S({\T})$ is finite.
\end{lem}

\begin{proof} Let  $K$ be a finite Galois extension of $F$ such that
$(T_F)_K$ is a split torus over $K$. It is easy to see that the
kernel of the canonical map
$\Sha_S({\T})\to \Sha_{S_K}({\T}\times \go_{K,S_K})$,
where $S_K$ is the set of primes of $K$ above $S$, is contained
in  $H^1(K/F, {\T}(\go_{K,S}))$.
Since ${\T}(\go_{K,S})\subseteq \T(\go_{K, S_1})$ is a finitely generated abelian group,
$H^1(K/F, {\T}(\go_{K,S}))$ is finite. Therefore the finiteness of
$\Sha_{S_K}({\T}\times_{\go_S} \go_{K,S})$ implies that
of $\Sha_S({\T})$ and we can assume that
$T_F$ is a split torus over $F$.
Enlarging $S$ if necessary (which will increase $\Sha_S(\T)$),
we can suppose that $S=S_1$ and
${\T}\cong \G_{m,\go_S}^d$, and even that $d=1$.

The short exact sequence
$$ 0\to {\T}(\bar{\go}_S)\rightarrow T_F(\bar{F})\rightarrow  T_F(\bar{F})/{\T}(\bar{\go}_S) \rightarrow 0 $$
gives the diagram of the following long exact sequence
$$
\begin{CD}
@.  T_F(F)/{\T}(\go_S) @>>> (T_F(\bar{F})/{\T}(\bar{\go}_S))^{G}
  @>>> H^1(F,{\T})@>>> {0} \\ 
@.  @VVV  @VVV  @VVV \\ 
0 @>>>   T_F(F_v)/{\T}(\go_v)  @>>> (T_F(\bar{F}_v)/{\mathbf
     T}(\bar{\go}_v))^{G_v} @>>> H^1(F_v,{\T})@>>>
   {0} 
\end{CD}$$
for any $v\not\in S$ by Galois cohomology, where $G=\Gal(\bar{F}/F)$
and $G_v=\Gal(\bar{F}_v/F_v)$. Since  for all $v\notin S$, $G$ acts
transitively on the primes of $\bar{\go}_S$ dividing $v$, one has
$$ \Ker ((T_F(\bar{F})/{\T}(\bar{\go}_S))^{G} \to \prod_{v\not \in S} (T_F(\bar{F}_v)/{\T}(\bar{\go}_v)))=0.$$
Note that the canonical map $(T_F(\bar{F})/{\T}(\bar{\go}_S))^{G}
\to \prod_{v\not \in S} (T_F(\bar{F}_v)/{\T}(\bar{\go}_v))$ takes
values in the direct sum. Thus by the snake lemma, and because ${\T}=\G_{m, \go_S}$,
$$\Sha_S({\T})\hookrightarrow
T_F(F)\backslash \bigoplus_{v\notin S} (T_F(F_v)/{\T}(\go_v))\cong \Pic(\go_S)$$
is finite.
\end{proof}

Let ${\T}$ be a commutative group scheme separated of finite type over $\go_{S_1}$ satisfying condition (\ref{split}) as before.
Consider the projective systems  $({\T}(\go_{S_1})/N{\T}(\go_{S_1}))_N$
and $(\Sel^{(N)}_S(F, {\T}))_N$, where for any pair of natural
integers $N \mid N'$, the transition map of the first system is
the canonical quotient map, and that of the second system is
 given by the multiplication-by-$N'/N$ map
$$\Sel^{(N')}_S(F, {\T})\xrightarrow{N'/N} \Sel^{(N)}_S(F, {\T}) . $$
Consider the limits
$$\widehat{{\T}(\go_{S_1})}:=\varprojlim_N {\T}(\go_{S_1})/N{\T}(\go_{S_1}) \quad 
\text{and} \quad \widehat{\Sel_S(F, {\T})}:=\varprojlim_{N} \Sel^{(N)}_S(F, {\T}).$$

\begin{cor}\label{T-S} With the above notations, we have a canonical isomorphism
$$ \widehat{{\T}(\go_{S_1})}\cong \widehat{\Sel_S(F, {\T})}.$$
\end{cor}

\begin{proof} This follows from the short exact sequence (\ref{short}) and Lemma \ref{sha}. \end{proof}

The next lemma is proved by similar arguments to those of Lemmas
1.3-1.5 in \cite{WX}.

\begin{lem} \label{h-k} There is a positive integer $h$ depending only on $\T$ such that
$${\T}(\go_{S_1})\cap \Big(hN\prod_{v\not\in S_1}{\T}(\go_v)\Big) \subseteq N {\T}(\go_{S_1}) $$ for any positive integer $N$.
\end{lem}

\begin{proof} Let $K$ be a finite Galois extension of $F$ such that $T_F={\T}\times_{\go_{S_1}} F$ splits over $K$, let $S'$ be the places of $K$
above $S_1$ and let ${\T}'={\T}\times_{\go_{S_1}} \go_{S'}$.
Let $G_K$ be the absolute Galois group of $K$. Then
${\T}'[N](\bar{\go}_{S'})$ is a sub-$G_K$-module of
$T_F[N](\bar{K})=\mu_{N,K}^d(\bar{K})$.
It follows that ${\T}'[N](\bar{\go}_{S'})$ is isomorphic to
a product of $\mu_{m, K}$'s.
We have the following commutative diagram
$$\begin{CD}
 H^1(K/F,{\T}[N](\go_{S'})) @>>> H^1(F, {\T}[N]) @>>> H^1(K, {\T}'[N]) \\
@.  @VVV  @VVV   \\
@.  \prod_{v\not\in S_1} H^1(F_v, {\T}[N])  @>>> \prod_{w|v, v\not \in S_1} H^1(K_w, {\T}'[N]).
\end{CD}$$
As $H^1(K/F,{\T}[N](\go_{S'}))$ is killed by $\gcd([K:F], r_K)$, where $r_K$ is the number of roots of unity inside $K$, and the
kernel of the last vertical arrow is killed by $2$ (\cite{NSW2}, Theorem (9.1.9)(ii)),
the kernel
$$ \Ker\big(H^1(F, {\T}[N])\rightarrow \prod_{v\not\in S_1} H^1(F_v, {\T}[N])\big) $$
is killed by $c:=2\gcd([K:F], r_K)$ for any positive integer $N$.
Let $h=2 c r_K$ which depends only on $\T$.

For any $x\in {\T}(\go_{S_1})\cap \big(h N\prod_{v\not\in S_1}{\T}(\go_v)\big)$, the following commutative diagram of Kummer exact sequences
$$
\begin{CD}
   {\T}(\go_{S_1}) @>{\cdot h N}>> {\T}(\go_{S_1}) @>>> H^1(F, {\T}[h N])  \\
@VVV  @VVV  @VVV \\
 \prod_{v\not\in S_1} {\T}(\go_v)@>{\cdot h N}>> \prod_{v\not\in S_1} {\T}(\go_v) @>>> \prod_{v\not\in S_1} H^1(F_v, {\T}[h N])
\end{CD}$$
implies that there exists $y\in {\T}(\go_{S_1})$ such that $c x= h N y$. Let
$$z=x-2r_K Ny \in {\T}(\go_{S_1}).$$
We have $$z\in 2r_KN \prod_{v\notin S_1}{\T}(\go_v)\subseteq
2r_KN\prod_{w\notin S'}T_K(K_w)$$ where $T_K:=T_F\times_F K$. By \cite{NSW2}, Theorem (9.1.11)(i), there exists $\xi\in T_K(K)=(K^\times)^{d}$ such that $z=r_K N \xi $. As
$z$ is a $c$-torsion point, $\xi$ is also a torsion point. Since $r_K$ kills all roots of unity in $K$, $z$ is trivial and $x=2r_KNy \in N {\T}(\go_{S_1})$. The proof is complete. \end{proof}

\begin{prop} \label{equiv}  Let $\overline{{\T}(\go_{S_1})}$ be the topological closure of ${\T}(\go_{S_1})$ inside $T_F(\A_F^S)$. Then the natural homomorphisms
$${\T}(\go_{S_1})/N {\T}(\go_{S_1}) \rightarrow  \overline{{\T}(\go_{S_1})}/(\overline{{\T}(\go_{S_1})}\cap N (\prod_{v\in S_1\setminus S} T_F(F_v) \times \prod_{v\not\in S_1}{\T}(\go_v))$$
induce an isomorphism of topological groups $\widehat{{\T}(\go_{S_1})} \cong \overline{{\T}(\go_{S_1})}$.
In particular, the inclusion $\T(\go_{S_1})\subset \overline{\T(\go_{S_1})}$ induces an equality of the torsion parts
$$ \T(\go_{S_1})_{\mathrm{tors}}=(\overline{\T(\go_{S_1})})_{\mathrm{tors}}.$$
\end{prop}

\begin{proof} Let
$$ A_N= ({\T}(\go_{S_1})\cap N (\prod_{v\in S_1\setminus S} T_F(F_v) \times \prod_{v\not\in S_1}{\T}(\go_v))/N {\T}(\go_{S_1}) $$
for any positive integer $N$. Since ${\T}(\go_{S_1})$ is a finitely generated abelian group, the quotient $\T(\go_{S_1})/N\T(\go_{S_1})$ is finite for all $N$. As the canonical map
$$\T(\go_{S_1})/N\T(\go_{S_1})\to  \overline{{\T}(\go_{S_1})}/(\overline{{\T}(\go_{S_1})}\cap N (\prod_{v\in S_1\setminus S} {T_F}(F_v) \times \prod_{v\not\in S_1}\T(\go_v))$$
has dense image, it is surjective.
Therefore one has the following exact sequence
 $$0\rightarrow A_N \rightarrow {\T}(\go_{S_1})/N {\T}(\go_{S_1}) \rightarrow  \overline{{\T}(\go_{S_1})}/(\overline{{\T}(\go_{S_1})}\cap N (\prod_{v\in S_1\setminus S} T_F(F_v) \times \prod_{v\not\in S_1}\T(\go_v))\rightarrow 0 $$
for any positive integer $N$. Moreover,
the inverse system $(A_N)_N$ satisfies the Mittag-Leffler condition by the
finiteness of $A_N$. By taking the inverse limits, one obtains the exact
sequence
 $$ 0\rightarrow \varprojlim_{N} A_N  \rightarrow \widehat{{\T}(\go_{S_1})} \rightarrow \widehat{\overline{{\T}(\go_{S_1})}}=\overline{{\T}(\go_{S_1})} \rightarrow 0, $$
the last equality follows from \cite{NS}, Theorem 1.1. On the other hand, $\varprojlim_{N} A_N =0$ by Lemma~\ref{h-k} and the the desired
isomorphism is proved. The equality on the torsion parts then follows easily from
the fact that $\T(\go_{S_1})$ is a finitely generated abelian group.
\end{proof}

Recall that in the Galois cohomology groups, ${\T}[N]$ stands
for the Galois module $\T[N](\bar{\go}_{S_1})$.

\begin{lem}\label{cohtor} There exists a positive integer $c$ depending only on $\T$ such that $c$ kills all $H^1(F_N/F, {\T}[N])$ for all positive integers $N$, where $F_N=F({\T}[N])$ is the smallest extension of $F$ such that $\T[N](\bar{\go}_{S_1})\subseteq T_F(F_N)$.
\end{lem}

\begin{proof} Let $K$ be a finite Galois extension of $F$ containing $F_N$. Then $\Gal(K/F_N)$ acts trivially on $\T[N](\bar{\go}_{S_1})$ and
we have an exact sequence
$$ 0\to H^1(F_N/F, {\T}[N])\to H^1(K/F,{\T}[N]) \to
H^1(K/F_N, {\T}[N]).$$
Taking $K$ for the compositum of $F_N$ with the splitting field
of the torus $T_F$, we are reduced to the case when $T_F$ is split.

Define the Galois module $\nu_N$ by the exact sequence
$$0\to {\T}[N](\bar{\go}_S) \to T_F[N](\bar{F})=\mu_{N,F}^d \to \nu_N \to 0.$$
Then one has
$$ \nu_N(F) \rightarrow H^1(F(\mu_N)/F, {\T}[N])\rightarrow H^1(F(\mu_N)/F, \mu_N^d) $$ and
$$ 0\rightarrow H^1(F_N/F, {\T}[N]) \rightarrow  H^1(F(\mu_N)/F, {\T}[N]) $$
by Galois cohomology. Let $r_F$
be the number of roots of unity inside $F$. Then $\nu_N(F)$ is killed by $r_F$
and one only needs to prove that $H^1(F(\mu_N)/F, \mu_N)$ is killed by a positive integer which is independent of $N$.

Let $N=\prod_{i=1}^t p_i^{m_i}$ be the prime factorization of $N$. Then one has
$$ H^1(F(\mu_N), \mu_N) = \prod_{i=1}^t H^1(F(\mu_N)/F, \mu_{p_i^{m_i}}) $$ and the exact sequence
$$ 0\rightarrow H^1(F(\mu_{p_i^{m_i}})/F, \mu_{p_i^{m_i}}) \rightarrow H^1(F(\mu_N)/F, \mu_{p_i^{m_i}}) \rightarrow H^1(F(\mu_N)/F(\mu_{p_i^{m_i}}), \mu_{p_i^{m_i}})^{G_i} $$ where $G_i=\Gal(F(\mu_{p_i^{m_i}})/F)$. Since $$ H^1(F(\mu_N)/F(\mu_{p_i^{m_i}}), \mu_{p_i^{m_i}})^{G_i}= \Hom(\Gal(F(\mu_N)/F(\mu_{p_i^{m_i}})), \mu_{p_i^{m_i}}(F)), $$ one concludes that $H^1(F(\mu_N), \mu_N)$ is killed by $2r_F$ by
\cite{NSW2}, Proposition 9.1.6.
\end{proof}

The following lemma is an application of the Chebotarev density theorem.
Recall that Equation~\eqref{sel} gives a canonical map
$$\Sel^{(N)}_S(F, \T)\to \T(\go_v)/N\T(\go_v)$$ for all $v\notin S_1$.

\begin{lem}\label{density} Let $c$ be a positive integer as in Lemma \ref{cohtor}.
Let $Q\in \Sel^{(N)}_S(F,{\T})$ and let $n$ be the order of $cQ$ inside
$\Sel^{(N)}_S(F,{\T})$.
Then the density of the following set
$$ \{ v\not\in S_1 \mid v \ \text{splits completely in $F_N/F$ and the image of $Q$
in ${\T}(\go_v)/N{\T}(\go_v)$ is $0$} \} $$
inside $\Omega_F$ is less than or equal to $1/(n\cdot [F_N:F])$, where $F_N=F({\T}[N])$.
\end{lem}
\begin{proof} As $\Sel^{(N)}_S(F,{\T})\subseteq
\Sel^{(N)}_{S_1}(F,{\T})$, one can replace $S$ with $S_1$ and suppose that
$S=S_1$. Let $S'$ be the primes of $F_N$ above the primes of $S$ and let
$$\Sel^{(N)}_S(F_N,{\T})=\Ker(H^1(F_N, \T)\to \prod_{w\notin S'} H^1(F_{N,w}, \T)).$$
The restriction map of Galois cohomology induces a homomorphism
$$\Sel^{(N)}_S(F,{\T}) \xrightarrow{\res}  \Sel^{(N)}_S(F_N,{\T}) $$
whose kernel is killed by $c$ by construction (Lemma~\ref{cohtor}). This implies that
the order of the image of $Q$ in $\Sel^{(N)}_S(F_N,{\T})$ is a multiple of $n$. Since
$$ \Sel^{(N)}_S(F_N,{\T}) \subseteq H^1(F_N, {\T}[N])=\Hom_{\mathrm{cont}} (\Gal(\bar{F}/F_N), {\T}[N]), $$
$Q$ gives rise to a
homomorphism $\alpha \in \Hom_{\mathrm{cont}}(\Gal(\bar{F}/F_N), {\T}[N])$
of order a multiple of $n$.
Let $L=\bar{F}^{\Ker(\alpha)}$. This is a Galois
extension of $F_N$ of degree a multiple of $n$.

If $v\not\in S$ splits completely in $F_N/F$ and $w$ is a prime in $F_N$ above $v$, then one has the following commutative diagram with injective row maps
$$
\begin{CD}
   \Sel^{(N)}_S(F_N,{\T}) @>>> \Hom_{\mathrm{cont}}(\Gal(\bar{F}/F_N), {\T}[N]) \\
    @VVV @VVV \\
 {\T}(\go_v)/N {\T}(\go_v) \cong {\T}(\go_{w})/N {\T}(\go_{w}) @>>> \Hom_{\mathrm{cont}}(\Gal(\overline{F_{N,w}}/ F_{N,w}),{\T}[N])
\end{CD}$$
where $\overline{F_{N,w}}$ is the algebraic closure of $F_{N,w}$. So the image of $Q$ in ${\T}(\go_v)/N {\T}(\go_v)$ is trivial if and only if
the decomposition group at any prime of $L$ above $w$ is trivial. The latter is equivalent to $w$ being split in $L/F_N$ completely. Therefore the set of primes of $F$ we consider is the same as the set of primes in $F$, not in $S$, and which split completely in $L$. By the Chebotarev density theorem, the density of this set is at most
$$\frac{1}{[L:F]}=\frac{1}{[L:F_N]\cdot [F_N:F]} \leq \frac{1}{n\cdot [F_N:F]}$$ and the proof is complete.\end{proof}

The main result of this section is the following theorem.

\begin{thm}\label{inters} If $Z$ is a finite closed subset of a torus
$T_F$ over $F$, then
$$ \overline{T_F(F)}^{S} \cap Z(\A_F^S) = Z(F) $$
where $\overline{T_F(F)}^S$ is the topological closure of $T_F(F)$ inside $T_F(\A_F^S)$.
\end{thm}

\begin{proof} We first reduce to the case $Z(F)=Z(\bar{F})$
using a similar trick to the proof of \cite{PV}, Proposition 3.9.
Consider a finite extension $K/F$ such that $Z(K)=Z(\bar{F})$.
Let $S'$ be the primes of $K$ above $S$. If one can show that
$$\overline{T_F(K)}^{S'} \cap Z(\A_K^{S'}) = Z(K),$$
then
$$Z(F)\subseteq \overline{T_F(F)}^S \cap  Z(\A_F^S) \subseteq
(\overline{T_F(K)}^{S'} \cap Z(\A_K^{S'}))\cap  Z(\A_F^S)= Z(K) \cap
Z(\A_F^S)=Z(F),$$
the last equality being proved similarly to Lemma 3.2 in \cite{PV},
and the desired equality will hold over $F$.

Since $Z(F)$ is finite, there exists a finite subset $S_1\supset S$ such that $T$ extend to an affine group scheme $\T$ of finite type over
$\go_{S_1}$ and $Z(F) \subset {\T}(\go_{S_1})$ and the condition (\ref{split}) is satisfied.
One only needs to show that $$\overline{{\T}(\go_{S_1})} \cap Z(\A_F^S) \subseteq Z(F)$$ where $\overline{{\T}(\go_{S_1})}$ is the topological closure of ${\T}(\go_{S_1})$ inside $T_F(\A_F^S)$. Indeed, the left-hand side contains
$\overline{T_F(F)}^{S} \cap Z(\A_F^S)$ as we saw in the proof of Corollary \ref{true}.

For any $$(x_v)_{v\not\in S}\in  \overline{{\T}(\go_{S_1})} \cap Z(\A_F^S) , $$ we claim that there is $z\in Z(F)$ such that $(x_v)_{v\not\in S}-z$ is torsion. Suppose
the contrary. Then all $(x_v)_{v\not\in S}-z$ with $z\in Z(F)$ are of infinite order. Fix a positive integer $n$ with $n>\sharp Z(F)$. Use Proposition~\ref{equiv} and
Corollary~\ref{T-S} to identify $\overline{\T(\go_{S_1})}$ to a subgroup
of $\widehat{\Sel_S(F, \T)}$. 
By the same argument as \cite{St}, Lemma 3.5, there is a positive integer $N$ such that the order of $c((x_v)_{v\not\in S}-z)$ in $\Sel^{(N)}_S(F,{\T})$ is at least $n$ for all $z\in Z(F)$, here $c$ is the positive integer defined Lemma~\ref{cohtor}. By Lemma \ref{density}, the density of primes $v\in \Omega_F\setminus S_1$ such that $v$ splits completely in $F_N/F$ and at least one of $(x_v)_{v\not\in S}-z$ is trivial in ${\T}(\go_v)/N {\T}(\go_v)$ for some $z\in Z(F)$ is at most $$ \frac{\sharp Z(F)}{n[F_N:F]}< \frac{1}{[F_N:F]} .$$ This implies that there is $v_0\not \in S_1$ such that $v_0$ splits completely in $F_N/F$ and none of $x_{v_0}-z$ is trivial in ${\T}(\go_{v_0})/N {\T}(\go_{v_0})$ for all $z\in Z(F)$. In particular, one has $x_{v_0}\neq z$ for all $z\in Z(F)=Z(F_v)$. This contradicts the hypothesis $(x_v)_{v\not\in S}\in Z(\A_F^S)$.

Since the torsion subgroup of $\overline{{\T}(\go_{S_1})}$ is the same as the torsion subgroup ${\T}(\go_{S_1})_{\mathrm{tor}}$ of ${\T}(\go_{S_1})$ by Proposition \ref{equiv}, one concludes that
$$(x_v)_{v\not\in S} \in Z(F)+{\T}(\go_{S_1})_{\mathrm{tor}} \subseteq {\T}(\go_{S_1})\subset T_F(F) .$$ The result follows from $Z(\A_F^S)\cap T_F(F)=Z(F)$. \end{proof}

Next we make a ``$S$-adelic Mordell-Lang conjecture'' for tori  in the same spirit of Question 3.12 in
\cite{St} which was proved for abelian varieties over a global function field in \cite{PV}. 

Let $X_F$ be a hyperbolic affine curve contained in ${\mathbf P}^1_F$ and let $j: X_F\rightarrow T_F$ be a closed immersion as in Lemma \ref{embed-T} with  a torus $T_F$. 
\begin{conjecture} If $G$ is a finitely generated subgroup of $T_F(F)$ and $\overline{G}^S$ is the topological closure of $G$ inside $T_F(\A_F^S)$, then there is a 
finite subscheme $Z\subset X$ such that 
 $$ X(\A_F^S)  \cap \overline{G}^S \subseteq Z(\A_F^S).$$
\end{conjecture} 

The Harari-Voloch conjecture over number fields will follow from this conjecture and Theorem \ref{inters} (see Question 3.12 in
\cite{St}).
\medskip

As an application of Theorem \ref{inters}, we show that Theorem 3 in \cite{HaVo} 
holds for any quasi-finite morphisms of rational hyperbolic curves 
(see Remarks~\ref{HV-equiv} and \ref{gen}).  

\begin{prop}\label{quasi-finite} Suppose $f: X_F\rightarrow Y_F$ is a quasi-finite morphism of curves over $F$ and $X_F$ is a rational curve inside ${\mathbf P}^1_F$ with complement of degree $\ge 2$. Fix a finite subset $S$ of primes of $F$ containing $\infty_F$. If $Y_F(\A_F^S)^{B_S(Y)}=Y_F(F)$, then $X_F(\A_F^S)^{B_S(X)}=X_F(F)$.
\end{prop}

\begin{proof} For any $x\in X_F(\A_F^S)^{B_S(X)}$, there is $y\in Y_F(F)$ such that $f(x)=y$ 
by the functoriality of the Brauer-Manin pairing 
and the assumption. Then $Z=f^{-1}(y)$ is a finite closed subscheme
over $F$ inside $X_F$ and $x\in Z(\A_F^S)$. Let $j : X_F\to T_F$ be the
closed immersion given by Lemma~\ref{embed-T}. Let 
$\overline{T_F(F)}^S$ be the closure of $T_F(F)$ in $T_F(\A_F^S)$.
By Lemma \ref{br} (2), see also Definition~\ref{B1S}, the immersion $X_F\to T_F$ induces a map
$$  X_F(\A_F^S)^{B_S(X)} \longrightarrow T_F(\A_F^S)^{B_{1,S}(T_F)}. $$
By Proposition \ref{eqhv} (1), one has
$$(T_F(\A_F^S))^{B_{1,S}(T_F)} = \overline{T_F(F)}^S,$$
hence
$$ x\in \overline{T_F(F)}^S\cap Z(\A_F^S) = Z(F)\subseteq X_F(F) $$ by Theorem \ref{inters}. The proof is complete.
\end{proof}

\begin{rem}\label{gen} Let $C_F$ be a projective smooth curve
of positive genus over $F$. In \cite{St}, Theorem 8.2, it is
proved that for any finite closed subset $Z$ of $C_F$, one has
$$ Z(\A_F)\cap C_F(\A_F)_{\bullet}^{\Br(C_F)}=Z(F),$$
where the subscript $\bullet$ means that at infinite places $v$,
$C_F(F_v)$ is replaced with its set of connected components
$\pi_0(C_F(F_v))$. 

Using the above result, J.-L. Colliot-Th\'el\`ene proved in \cite{CT-HV} 
that \cite{HaVo} Theorem 3, can be generalized to covers of
arbitrary degree, provided the group $B_S(.)$ is replaced
by the whole Brauer group $\Br(.)$. Our proof of 
Proposition~\ref{quasi-finite} is based on similar 
arguments.  If $X_F$ is an open subvariety of $C_F$ containing
$Z$ and if the analogue equality
$$ Z(\A_F^S)\cap X_F(\A_F)^{B_S(X_F)}=Z(F)$$
holds, then one can remove the rationality hypothesis on
$X_F$ in Proposition \ref{quasi-finite}. 
\end{rem}

\begin{cor}\label{hv-ex2} Let $D$ be a reduced effective divisor in $\mathbf P^1_\Q$ 
supported in at least two points. Let $X_\Q$ be the complement of $D$ and 
let $S=\{\infty_\Q\}$. Then 
$$X_\Q(\Q)=X_\Q(\A_\Q^S)^{B_S(X_\Q)}.$$
In particular, Harari-Voloch conjecture holds for $X_\Q$ and $S$ if $\deg D\ge 3$
(equivalently, if $X_\Q$ is hyperbolic). 
\end{cor}

\begin{proof} 
Let $P_1, P_2$ be two points in $D$. There exists a non-constant invertible 
element $f\in \cO(X_\Q)^\times$. The finite morphism $\mathbf P^1_\Q\to 
\mathbf P^1_\Q$ defined by $f$ induces a quasi-finite morphism
$X_\Q\to Y_\Q:=\G_{m, \Q}$. As $$Y_\Q(\A_\Q^{S})^{B_S(Y_\Q)}=Y_\Q(\Q)$$ 
(e.g. by Proposition~\ref{eqhv} (2) and Theorem~\ref{stp} (2)), 
the same equality holds for $X_\Q$ by Proposition~\ref{quasi-finite}. 
\end{proof}

\noindent{\bf Acknowledgements.}
We would like to thank J.-L. Colliot-Th\'el\`ene, D. Harari, Y. Liang
and F. Voloch for helpful discussions. We would also like to thank Liang-Chung
Hsia for drawing our attention to \cite{Sun}. We thank the referee for 
a careful reading of the manuscript and for pointing out some inaccuracies.
The first named author thanks Capital
Normal University, where part of this work was done, for its support. The second 
named author is supported by the ALGANT program in Universit\'e de Bordeaux, MPI for mathematics at Bonn from 
Sep.-Oct.2014 and NSFC grant no. 11031004.


\end{document}